\documentclass[12pt]{article}
\usepackage{amsmath,amsthm}
\usepackage{graphicx}
\usepackage{subfigure}
\usepackage{amssymb,latexsym}

\DeclareGraphicsExtensions{.pdf,.png,.jpg}

\newtheorem{theorem}{Theorem}[section]
\newtheorem{lemma}[theorem]{Lemma}

\newtheorem{corollary}[theorem]{Corollary}
\newtheorem{definition}[theorem]{Definition}
\theoremstyle{remark}
\newtheorem{remark}[theorem]{Remark}

\newcommand{\R}{\mathbb R}

\renewcommand{\a}{\alpha}

\newcommand{\sech}{{\rm sech~}}
\newcommand{\eps}{\varepsilon}
\newcommand{\mm}{{\mu\over 2}}

\renewcommand{\gg}{\gamma}
\newcommand{\ggh}{\hat\gamma}

\newcommand{\vb}{\vec\beta}

\newcommand{\va}{\vec\alpha}

\begin{document}

\title{Smooth parametric dependence of asymptotics of the semiclassical focusing NLS}

\author{Sergey Belov \footnote{ Department of Mathematics, Rice University, Houston, TX 77005, e-mail: belov@rice.edu},
Stephanos Venakides \footnote{ Department of Mathematics, Duke University, Durham, NC 27708, e-mail: ven@math.duke.edu.
SV thanks NSF for supporting this work under grants NSF DMS-0707488 and NSF DMS-1211638.}}

\date{}
\maketitle

\begin{abstract}

We consider the one dimensional focusing (cubic) Nonlinear Schr\"odinger equation (NLS)
in the semiclassical limit with  exponentially decaying complex-valued initial data, whose phase is multiplied by a real parameter.
We prove smooth dependence of the asymptotic solution on the parameter. Numerical results
supporting our estimates of important quantities are presented.

\end{abstract}

\section{Introduction}

 We consider
the semiclassical focusing nonlinear Schr\" odinger (NLS) equation
%\cite{CT, Lyng,KMM_book_03,MillerKamvissis}
\begin{equation}
i\eps\partial_t q+\eps^2\partial_x^2 q +2|q|^2q=0
\label{eq:1.1_NLS}
\end{equation}
with the initial data
\begin{equation}
q(x,0)=A(x)e^{\frac{i\mu}{\eps}S(x)},\ \ \ A(x), S(x)\in\mathbb{R},\ \ \ \mu\ge 0,  \label{IC_NLS}
\end{equation}
in the asymptotic  limit $\eps\to 0$.  Eq. \eqref{eq:1.1_NLS} is a well-known  integrable system \cite{ZS}, and a lot of work has been done on the above initial value problem (see below).  The focus of the present study  is on the parameter $\mu$ in the  the exponent of the initial data.
%Its  significance stems from the way it affects the underlying Lax spectrum, which (due to the integrability of the system)  is preserved as the system evolves in time.
For the specific data
\begin{equation}
 A(x)=-\sech x, \ \ \ \  \  S'(x)=-\tanh x, \ \ \ \ \     \mu\ge 0,  \label{IC_NLS_special}
\end{equation}
studied earlier \cite{KMM_book_03, TVZzero}, the solution
undergoes a transition at $\mu=2$.  When $\mu< 2$,   the Lax
spectrum  contains discrete eigenvalues numbering
$O\left(\frac{1}{\varepsilon}\right)$, each eigenvalue giving rise to a
soliton in the solution, which thus consists of both a radiative
and a solitonic part.  When $\mu\ge 2$, the spectrum is  purely
continuous and the solution is purely radiative  (absence of
solitons). We prove that the local wave parameters   (branchpoints
of the Riemann surface that represents the asymptotic solution
locally in space-time), vary smoothly with $\mu$, even at the
critical value  $\mu=2$. Indeed,  numerical experiments have shown
absence of any noticeable transition in the behavior of the
branchpoints at the critical value \cite{MillerKamvissis}.
Theorem \ref{thm:perturb_NLS} establishes this fact
rigorously.

The reason $\mu$ deserves a special attention as a perturbation
parameter is twofold: 1. At the value of $\mu=2$ there is a phase
transition of the nature of the solution (there is a solitonic
part when $\mu<2$, see below). Perturbing $\mu$ across this value
allows continuing the validity of a rigorously derived asymptotics
\cite{TVZlong2} from the region $\mu\ge 2$ to the region $\mu<2$.
Ab initio derivation of such asymptotics in the region $\mu<2$
would be technically more demanding. 2. $\mu$ is a singularity of
the RH contour in a way that cannot be remedied by contour
deformations. Such a difficulty is absent when the perturbation
parameters space or time variables $x$ and $t$. Indeed the methods
of \cite{TV_det} and \cite{TV_param} are applied in this work and amongst the
surprises which allow the methods to apply is a collection of
explicit formulae for dependence in $\mu$ summarized in Lemmas
3.2-3.4.

Essential mathematical difficulties are encountered in the
solution of the  initial value problem \eqref{eq:1.1_NLS} and
\eqref{IC_NLS}  in general and \eqref{eq:1.1_NLS} and
\eqref{IC_NLS_special} in particular.
\begin{enumerate}
\item The  calculation
of the scattering data at $t=0$ is extremely delicate as seen from
the work of Klaus and Shaw \cite{KlausShaw}.
 \item The linearizing Zakharov-Shabat eigenvalue problem  \cite{ZS} is {\it not
 selfadjoint}. This is in contrast to the selfadjointness of the initial value  problem for the small-dispersion  Korteweg-de Vries (KdV) equation, in which a systematic steepest descent procedure was developed by Deift, Venakides and Zhou \cite{DVZ1} for calculating the asymptotic solution (see also \cite{DVZ_longtimeKdV}).   The approach in \cite{DVZ1}   extended   the original steepest descent analysis of Deift and Zhou \cite{DZ1} for oscillatory Riemann-Hilbert problems, by adding to it  the $g$-function mechanism. A  systematic procedure then  obtained the KdV solution, which  consists of  waves that are fully nonlinear. These waves are typically modulated. In other words, the oscillations are rapid, exhibiting   wavenumbers  and frequencies in the small spatiotemporal scale,   that vary in the large scale in accordance to {\it modulation equations}.

\item The system of modulation equations in the form of a set of PDE for the NLS equation exhibits complex characteristics (Forest and Lee \cite{FL}).  Posed naturally as an initial   value problem, the system is thus {\it ill-posed} and   modulated NLS  waves are unstable. The  instability to the  large-scale spatio-temporal variation of the wave parameters  (modulational instability) is the primary source of problems in  nonlinear fiberoptical transmission, which is  governed by the NLS equation.
\end{enumerate}
 In spite of the modulational instability, there exist  initial data with a particular combination of $A$ and $S$ that evolve into a profile of modulated waves. The ordered structure of modulated nonlinear waves was first observed numerically by Miller and Kamvissis  \cite{MillerKamvissis}, for the  initial data of \eqref{IC_NLS_special} with  $\mu=0$ and values of $\eps$ that allowed them to implement the multisoliton NLS formulae. Miller and Kamvissis   observed the phenomenon of wave  breaking (see below) and the formation of more complex wave structures  past the break in this work.
Later numerical findings by Ceniseros and Tian \cite{CeniserosTian}, as well as by  Cai, Mc Laughlin (D.W.), Mc Laughlin (K. T-R) \cite{Cai} also detected the ordered structures.

These studies were followed by analytic work of Kamvissis, McLaughlin and Miller \cite{KMM_book_03}  for the same initial data  ($\mu=0$), the corresponding initial scattering data having been  earlier calculated explicitly by Satsuma and Yajima \cite{SatsYaj}. This work  set forth a procedure aiming at the analytic determination of the observed phenomena  that would practically extend the steepest descent  procedure  cited above  to  the non-selfadjoint case. Following a similar approach,
Tovbis, Venakides, Zhou, derived these phenomena rigorously from the initial data \eqref{IC_NLS_special} with $\mu>0$ \cite{TVZzero}; the initial scattering data  were previously obtained by Tovbis and Venakides \cite{TV_ScatData}.  The asymptotic calculation of the wave solution was with error of order $\eps$ for points in space-time that are off the break point and off the caustic curves (see below). In further work   \cite{TVZlong2, TVZ2},  the same authors also derived  the long-time behavior of the asymptotic solution for $\mu\ge 2$  and generalized  the  prebreak analysis to a wide class of initial data.  The detailed asymptotic behavior in a neighborhood of the first break-point was derived recently by Bertola and Tovbis \cite{BertTov1}.

The rigorous derivation of the mechanism of the second break  remains  an open problem. Using a combination of theoretical and numerical arguments, Lyng and Miller  \cite{Lyng} obtained significant insights for  initial data \eqref{IC_NLS_special} with  $\mu=0$  when  the solution is an  $N$-soliton, where $N=O\left(\frac{1}{\varepsilon}\right)$. In particular, they identified  a mathematical mechanism for the second break, which depends essentially on the discrete nature of the spectrum of the N-soliton and turns out to differ from the mechanism of the first break.

The asymptotic solution for {\it shock initial data},
\begin{equation}
 A=\mbox{constant}, \ \ \ \ S'(x)=\mbox{sign}x,  \ \ \ \ \mu>0,
\end{equation}
was derived globally in time  by Buckingham and Venakides \cite{BuckVen}.

The work of the present paper rely  on the determinant form of the modulation equations of NLS obtained  by Tovbis and Venakides  \cite{TV_det}. The modulation equations are transcendental equations, not differential equations, thus the modulational instability does not hinder the analysis.  Tovbis and Venakides   utilized the determinant form to study the variation of the asymptotic procedure  as  parameters of the Riemann-Hilbert problem, in particular the spatial and the time variables that are parameters in the Riemann-Hilbert problem analysis, change. They  proved \cite{TV_param} that, in the case of a regular break, the nonlinear steepest descent asymptotics can be ``automatically'' continued through the breaking curve (however, the expressions for the asymptotic solution will be different on the different sides of the curve).  Although the results are stated and proven for the focusing NLS equation, they can be reformulated for AKNS systems, as well as for the nonlinear steepest descend method in a more general setting. The present paper examines the variation of the procedure with respect to the parameter
$\mu$ and proves that the variation is smooth even as $\mu$ crosses the critical
value $\mu=2$.

\subsection{Background: $n$-phase waves, inverse scattering and the Riemann-Hilbert Problem}
In order to make the study accessible beyond the group of experts in the subject, we give an overview of our understanding of  the phenomenology of the time evolution of the semiclassical NLS equation and the mathematics that represents this phenomenology.

In the ideal (and necessarily unstable) scenario, in which  modulated wave profiles persist in space-time, so does the  separation in two  space-time scales.
In the large scale, a set of boundaries (breaking or caustic curves) divides the space-time  half-plane $(x,t)$, $t>0$ into regions. Inside each region, and in the leading order as $\eps \to 0$,
the solution  is an $n$-phase wave ($n=0,1,2,3,...$), with waveperiods and wavelengths in the small scale. The  wave parameters vary in the large scale. The increase in $n$ occurs typically as a new phase is generated  at a point in space-time, due, for example, to wave-breaking (more precisely, to avert wave-breaking) or to two existing phases coming together. The newly generated oscillatory  phase spreads in space with finite speed and the trace of its fronts in space-time  constitutes the set of breaking curves.

An $n$-phase NLS wave  is a
solution of $(\ref{eq:1.1_NLS})$ which exhibits a ``carrier" plane wave and $n$ nonlinearly  interacting wave-phases that control its oscillating amplitude.  The wave is characterized by a set of $2n+2$ real wave parameters:
 $n+1$ frequencies and $n+1$ wavenumbers. In the scenario discussed above, waves with  periods and wavelengths of order $O(\eps)$  constitute  the  the small space-time scale. The boundaries separating phases in space-time exist in the large scale, which
is of order $O(1)$. These boundaries play the role of nonlinear caustic curves.  The analytic wave profile  of an $n$-phase wave   is given explicitly  in terms of an elliptic ($n=1$) or hyperelliptic ($n>1$) Riemann theta function, derived from a compact Riemann surface of genus $n$. This is true not only for NLS but for most of the integrable wave equations studied. The $2n+2$ branchpoints of the Riemann surface
are the wave parameters of choice, that determine the $n+1$ frequencies and $n+1$ wavenumbers.  In the case of NLS, the $0$-phase wave is simply a plane wave.

The initial data \eqref{IC_NLS} have the structure of a
modulated $0$-phase wave. As $t$, increasing
from zero, reaches a value $t=t_{break}$, the $n=0$ initial phase
breaks at a caustic point in space-time. As described above, a wave-phase of higher  $n$ emerges then and spreads in space. As time increases, the endpoints of the spatial interval of existence of the new phase define the two caustic curves in space-time that
emanate from the break-point. The eventual breaking of this new phase, the so-called {\bf second break}.  The mechanism of the second break is fundamentally different from the one of the first.

As stated above, the {\bf  analytic description of $n$-phase waves} \cite{Belo} is in terms  of an $n$-phase Riemann theta function. This  is
an  $n$-fold Fourier series obtained by  summing
exp\{$2\pi i{\bf z}\cdot{\bf m}+\pi i(B{\bf m},{\bf m})\} $
over the multi-integer ${\bf m}$.  $B$ is an $n\times n$ matrix with
positive definite imaginary part that gives the series
exponential quadratic convergence.
In the case of NLS waves, the
 matrix $B$ arises from periods of the Riemann surface of the radical
\begin{equation} \label{radical}
R(z)=\left(\prod_{i=0}^{2n+1}(z-\alpha_i)\right)^\frac{1}{2}, \ \ \ \ \ \mbox{where} \ \ \  \a_{2j+1}=\overline\a_{2j};
\end{equation}
the elements of $B$ are  linear combinations of
the hyperelliptic  integrals $\oint z^k/R(z), \ k=0,1\cdots n-1$ along appropriate
closed contours on the Riemann surface of $R$ \cite{Belo}.
The series has
a natural quasi-periodic structure in the  $n$ complex arguments
${\bf z}=(z_1,\cdots z_n)$. Each ${z_j}$ is linear in $x$
and $t$  and represents a nonlinear phase of the wave; the wavenumbers and
 frequencies are expressed in terms
of hyperelliptic integrals of the radical $R$ and are thus
functions of the $\a_{2j}$, whose status as preferred parameters
is obvious from \eqref{radical}.

The semiclassical limit procedure  derives the emergent wave structure described above {\bf without} any a priori ansatz of such structure. The radical $R(z)$ and the  wave parameters arise naturally in  the procedure. As mentioned above,  these wave structures are modulated in space-time. In the large space-time scale, the branchpoints $\alpha_i$  vary  and their number experiences a jump across the breaking curves.  The  branchpoints are calculated from the modulation equations in determinant form (they are transcendental  equations, not partial differential equation, thus there is no ill-posedness at hand). The number of branchpoints is obtained with the additional help of sign conditions that are obtained in the procedure.

\noindent{\bf Overview of scattering, inverse scattering and the Riemann-Hilbert problem (RHP):}
The NLS was solved by Zakharov and Shabat \cite{ZS}, who discovered a Lax pair that linearizes NLS. The Lax pair refers to two ordinary differential operators, one in the spatial variable $x$ and the second in the time variable $t$.

The first operator of the Lax pair is a Dirac-type operator that is not selfadjoint. The corresponding eigenvalue problem (ZS)
  is a $2\times 2$ first order
{\bf linear} ODE, with  independent variable $x$. The NLS solution $q(x,t,\eps)$ plays the role  of a scatterer,  entering in the off-diagonal entries of the ODE matrix.  {\bf Scattering data}
are defined for those values of the spectral parameter $z$ that
produce bounded (ZS) solutions. This happens when $z$ is real
(these solutions are called  scattering states) and at the discrete
set of proper  eigenvalues $z_j$ (the eigenvalues come  in
complex-conjugate pairs; the  normalized  $L^2$ solutions are called   bound states).  The {\bf reflection coefficient} $r(z)$, $z\in\R$ provides a connection between   the asymptotic behaviors of the scattering states  as $x\to\pm\infty$. The {\bf norming
constants} $c_j$,  corresponding
to the  proper eigenvalues $z_j$, provide the asymptotic behavior of the bound states as $ x\to+\infty$.

The second operator of the Lax pair evolves the scattering states and the bound states in time and is again a $2\times 2$ linear ODE system. The holding of the NLS equation guarantees that this    evolution  involves the action of a time dependent unitary operator. As a result, the  spectrum of the first Lax operator remains constant in time and the scattering data evolve in time through multiplication by simple explicit exponential propagators. The continuous spectrum contributes  {\bf radiation}  to the solution of NLS. The bound states contribute  {\bf solitons}.

 Zakharov and Shabat \cite{ZS} developed the inverse scattering procedure
for deriving  $q(x,t)$ at any $(x,t)$ given the scattering data at $t=0$. In the  modern approach initiated by Shabat \cite{Shabat}, the procedure is recast into a {\bf matrix Riemann-Hilbert problem} (RHP) for a $2\times2$ matrix on the complex plane of the
spectral variable $z$. One needs to determine
the  matrix $m(z)$ that is analytic on the closed complex plane, off an  oriented contour $\Sigma$, that consists of the real axis and
of small circles surrounding the eigenvalues.
Modulo  multiplication of its columns by normalizing factors
$e^{\pm ixz/\eps}$, the  matrix $m(z)$ (the unknown of the problem at $t>0$)
is a judiciously specified fundamental matrix
 solution of the eigenvalue problem of the first operator of the Lax pair (Zakharov-Shabat eigenvalue problem). In order to determine the matrix $m(z)$, one is given a jump condition  on the contour  $\Sigma$, and  a normalization condition at $z\to\infty$,
\begin{equation}
  \label{m}
m(z)=
\left( \begin{array}{cc} m_{11}&m_{12}\\
m_{21}&m_{22}\end{array}\right)
\to \mbox{ Identity, as} \ z\to \infty; \ \ \ \ \ \ \ m_+(z)=m_-(z)V \ \mbox{when} \ z\in\Sigma.
\end{equation}
The subscripts $\pm$ indicate limits taken from the left/right of the contour. The $2\times2$ matrix $V=V(z,x,t,\eps)$, defined on the jump contour and referred to as the {\bf jump matrix}, is  nonsingular and encodes the scattering data (see below). The space-time variables $x,t$ (and the semiclassical parameter $\varepsilon$) are parameters in the problem.

The  solution  to \eqref{eq:1.1_NLS} is given by the simple formula
\begin{equation}
 q(x,t,\eps)=\lim_{z\to\infty}zm_{12}(z,x,t,\eps).
\end{equation}
The results and the calculations of this study are in the asymptotic limit of the semiclassical parameter   $\eps\to 0$.

\subsection{Background: The semiclassical limit $\eps\to 0$}

The Riemann-Hilbert approach is a major tool in the asymptotic analysis of integrable systems as
established with the discovery of the steepest descent method \cite{DZ1,DZ2} and its extension through the $g$-function mechanism \cite{DVZ1,TVZzero}.  The asymptotic
methods via the RHP approach also apply to orthogonal polynomial asymptotics
\cite{Deift_book, DKMVZ_strong}, and to random matrices
\cite{BDJ, DKMVZ_uniform, DJ, EM}.

 The semiclassical asymptotic analysis of the  highly oscillatory RHP  is similar in spirit to the steepest descent method for integrals in the complex plane. Throughout the analysis, the quantities $x$, $t$, and $\mu$ enter as parameters. The semiclassical analysis, performed with the aid of the $g$-function mechanism  \cite{DVZ1, TVZzero}, is constructive. An undetermined function $g(z)$ is introduced in the RHP through a simple transformation of the independent matrix variable of the RHP.   The contour of the RHP, itself an unknown,  is partitioned (in a way to be determined) into two types of interlacing subarcs. The jump-matrix is manipulated differently in the two subarc types.  In one of them  (main arcs) the jump matrix is factored in a certain way. In the other type (complementary arcs), it is factored differently or is left as is.  The $g$-function mechanism then   imposes  on appropriate entries of the jump matrix factors  the condition of  constancy in the complex spectral variable combined with boundedness  as $\eps\to 0$, while imposing decay as $\eps\to 0$ on other entries.  The constancy conditions are equalities, the decay conditions are sign conditions that act on exponents, forcing the decay of the corresponding exponential entries. Put together these conditions constitute a scalar RHP for the function $g(z)$ (or as below on its sister function $h(z)$). The contour of the RHP,  its partitioning and finally the functions $g(z)$ and $h(z)$ follow from the analysis of this scalar RHP.  This procedure allows the peeling-off of the leading order solution of the original matrix RHP and leaves
behind a matrix RHP for the error. This RHP is solvable with the aid of a Neumann series.

The  formulae for the conditions obtained through the $g$-function mechanism have an intuitive interpretation that arises from
(2D) potential theory in the complex plane of the spectral
parameter $z$. The main question is to determine the equilibrium
measure for an energy
functional \cite{Deift_book,KR_05, Simon_Szego_thm_book} that depends parametrically  on the variables $x$ and $t$. The  support of the measure depends on $x$ and $t$. For problems with self-adjoint Lax operators
({\it e.g.} Korteweg-de Vries equation in the small dispersion limit), the  support is on the real line (typically a set of intervals as in \cite{LaxLev1,LaxLev2,LaxLev3, Venak_reflec}). In a general non-selfadjoint case the supports are in the complex plane.  This is the case with the (focusing) NLS under study, whose spatial  Lax operator, the   Zakharov-Shabat system,  is of Dirac-type and is non-selfadjoint.

The conditions obtained through the $g$-function mechanism are exactly the variational conditions for the equilibrium measure.
Deriving these conditions rigorously as such is highly taxing, especially in the nonselfadjoint case. This is not needed though.
The conditions are used essentially as an Ansatz in the RHP. As long as the calculation confirms the Ansatz, the whole procedure is rigorous.

In the cases of the (focusing) NLS  in the semiclassical limit that have been worked out so far \cite{BKS_step, BuckVen, Lyng, TVZzero, TVZ2},  the support of the equilibrium measure  is a finite union of arcs in the complex plane with complex-conjugate symmetry.
Denote the end points of the ``main arcs''  (see below) as $\{\alpha_j\}_{j=0}^{N'}$ with some finite $N'\in\mathbb{N}$. The analysis leads naturally to the representation of  these points as the roots of a monic polynomial, the square root of which is exactly the radical in \eqref{radical}. This radical, a finite genus Riemann surface,  constitutes the passage to the periodic structure of the  local waveform. We refer to these endpoints henceforth as ``branchpoints''.
It is necessary  to establish the existence and the number of the branchpoints for each pair $(x,t)$
as well as the existence of the arcs,  which provide the leading contribution to the expression of the local waveform.

The approach to obtaining the asymptotic solution from the initial data  is to analyze the RHP for fixed $x$ and $t$, thus treating the space and time variables as parameters. The smooth dependence of the branchpoints $\alpha_j$
on the parameters $x$ and $t$ for semiclassical NLS with $q(x,0)=\sech(x)$, $\mu=0$ was studied in \cite{KMM_book_03} by considering
moment conditions (pp.148-162). A different approach is to start with local behavior near $\alpha_j$'s was applied in \cite{TV_det}
leading to formulae of the form
\begin{equation}
\frac{\partial\alpha_j}{\partial x}(x,t)=-\frac{2\pi i\ \frac{\partial K}{\partial x}(\alpha_j,\va,x,t)}{D(\va) \oint_{\ggh}\frac{f'(\zeta,x,t)}{(\zeta-\alpha_j(x,t))R(\zeta,\va)}d\zeta},
\label{eq:D_alpha_D_x_intro}
\end{equation}
\begin{equation}
\frac{\partial\alpha_j}{\partial t}(x,t)=-\frac{2\pi i\ \frac{\partial K}{\partial t}(\alpha_j,\va,x,t)}{D(\va) \oint_{\ggh}\frac{f'(\zeta,x,t)}{(\zeta-\alpha_j(x,t))R(\zeta,\va)}d\zeta},
\label{eq:D_alpha_D_t_intro}
\end{equation}
where $\va=\va(x,t)$.

In addition to \cite{TV_det}-\cite{TV_param} this work is similar in spirit to
\cite{KM_00} which pertains to random matrix theory. An analogous
quantity to the function $f$ is a potential.
Those authors put a strong topology on the set of allowable
potentials (a potential is analogous to the function $f$ in the
present work) and demonstrate that the so-called "regular case" is
generic - i.e. if you find a potential with fixed genus and all
other side-conditions are satisfied in a strict sense, then the
same genus holds for all potentials in a neighborhood.

In this study, we obtain the following main results.
\begin{enumerate}
 \item
We extend formulae
(\ref{eq:D_alpha_D_x_intro}-\ref{eq:D_alpha_D_t_intro}) to include  the dependence  the branchpoints and  of the contour $\ggh=\ggh(\mu)$ on the external parameter
$\mu$,
\begin{equation}
\frac{\partial\alpha_j}{\partial \mu}(x,t,\mu)=-\frac{2\pi i\ \frac{\partial K}{\partial \mu}(\alpha_j,\va,x,t,\mu)}{D(\va) \oint_{\ggh(\mu)}\frac{f'(\zeta,x,t,\mu)}{(\zeta-\alpha_j(x,t,\mu))R(\zeta,\va)}d\zeta},
\label{eq:D_alpha_D_mu_intro}
\end{equation}
where $\va=\va(x,t,\mu)$. We show that the dependence is smooth, meaning that the contour,
the jump matrix, and the solution of the scalar RHP evolve
smoothly (they are continuously differentiable) in
$\mu$.
\item We simplify
the expression for $\frac{\partial K}{\partial \mu}$ as (\ref{eq:DK/Dmu_function_NLS}).
\item
We show good agreement of  formula (\ref{eq:D_alpha_D_mu_intro}) with the dependence of the  branchpoints on the parameter $\mu$, obtained by the direct numerical solution of the system (\ref{eq:mod_K_def}) (see Fig. \ref{fig:evolutions}).
\item
We prove the preservation of genus of the asymptotic solution in an open interval of parameter
$\mu$. In particular, the genus is preserved ($0$ or $2$)
for all $x$ and $t>0$ for some open interval (which depends on $x$ and $t$) for $\mu<2$.
\end{enumerate}

\vskip.3cm

The paper is organized as the following: {\it Section 2} -
Definitions and prior results; {\it Section 3.1} -  Analyticity of
$f$ in $\mu$, and differentiability of the branchpoints
$\alpha_j=\alpha_j(\mu)$; {\it Section 3.2} - $\mu$ dependence of
quantities that appear in Theorem \ref{thm:perturb_NLS}; {\it
Section 3.3} -  Sign conditions and preservation of genus theorem
\ref{thm:pres_genus_NLS}; {\it Section 3.3} - Numerics; {\it
Section 4} - Appendix.

\section{Preliminaries}

\begin{figure}
\begin{center}
\includegraphics[height=5cm]{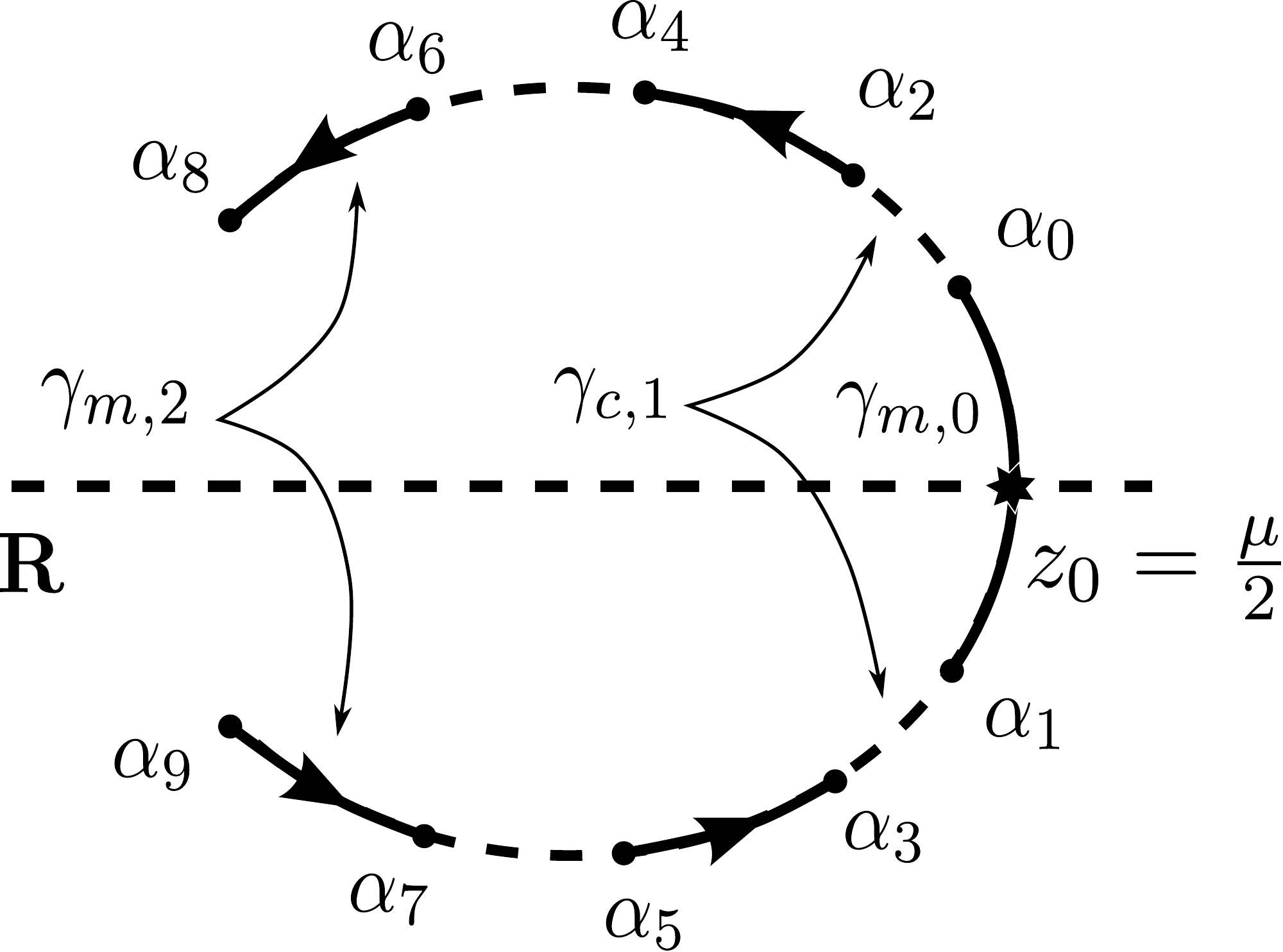}
\caption{\label{fig:RHP_contours} The RHP jump contour in the case of genus 4 with complex-conjugate symmetry in the notation of \cite{TVZzero}.}
\end{center}
\end{figure}

\begin{figure}
\begin{center}
\subfigure[Contour $\ggh$]{\includegraphics[height=6cm]{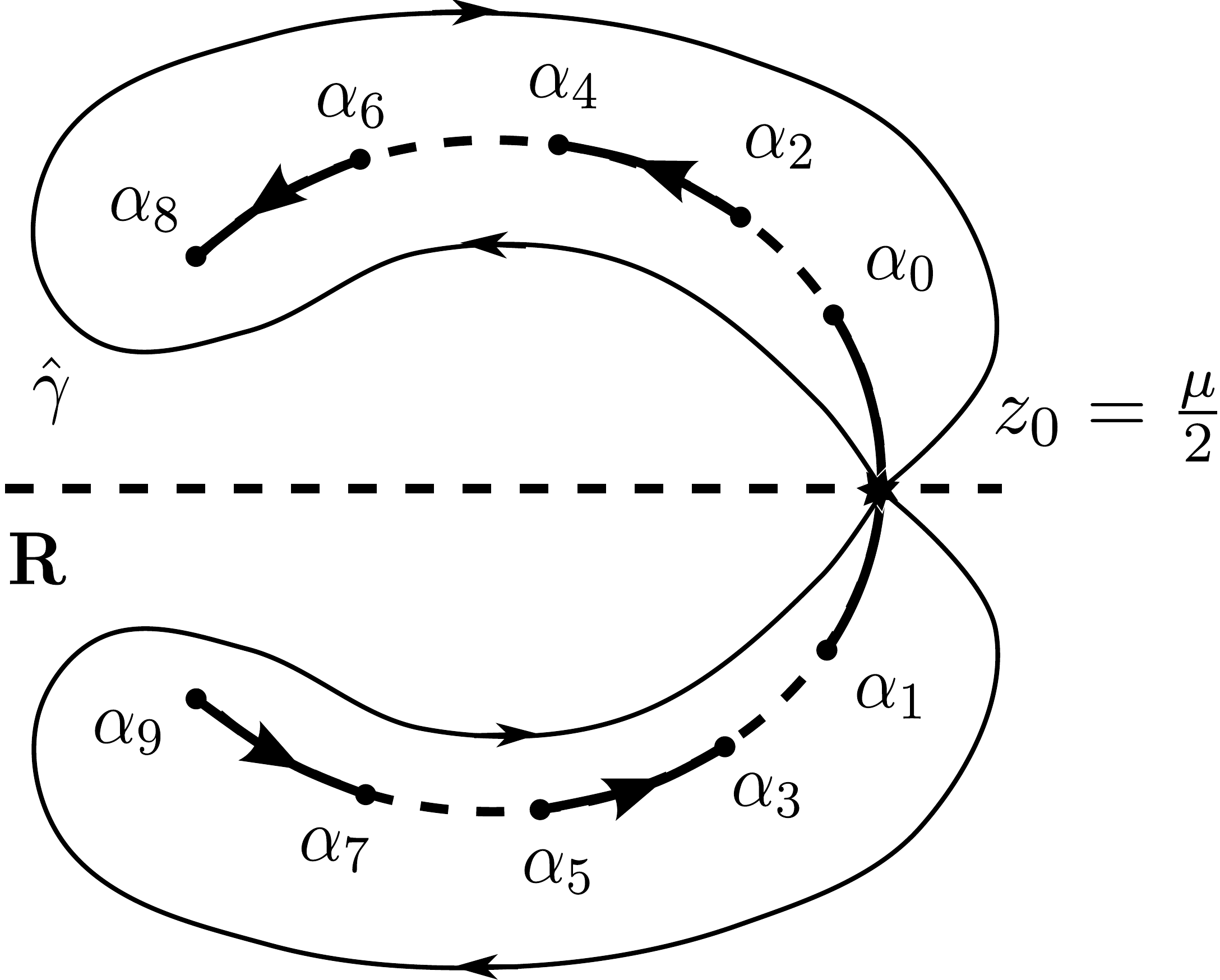}}
\subfigure[Contours $\ggh_{m,j}$ ]{\includegraphics[height=6cm]{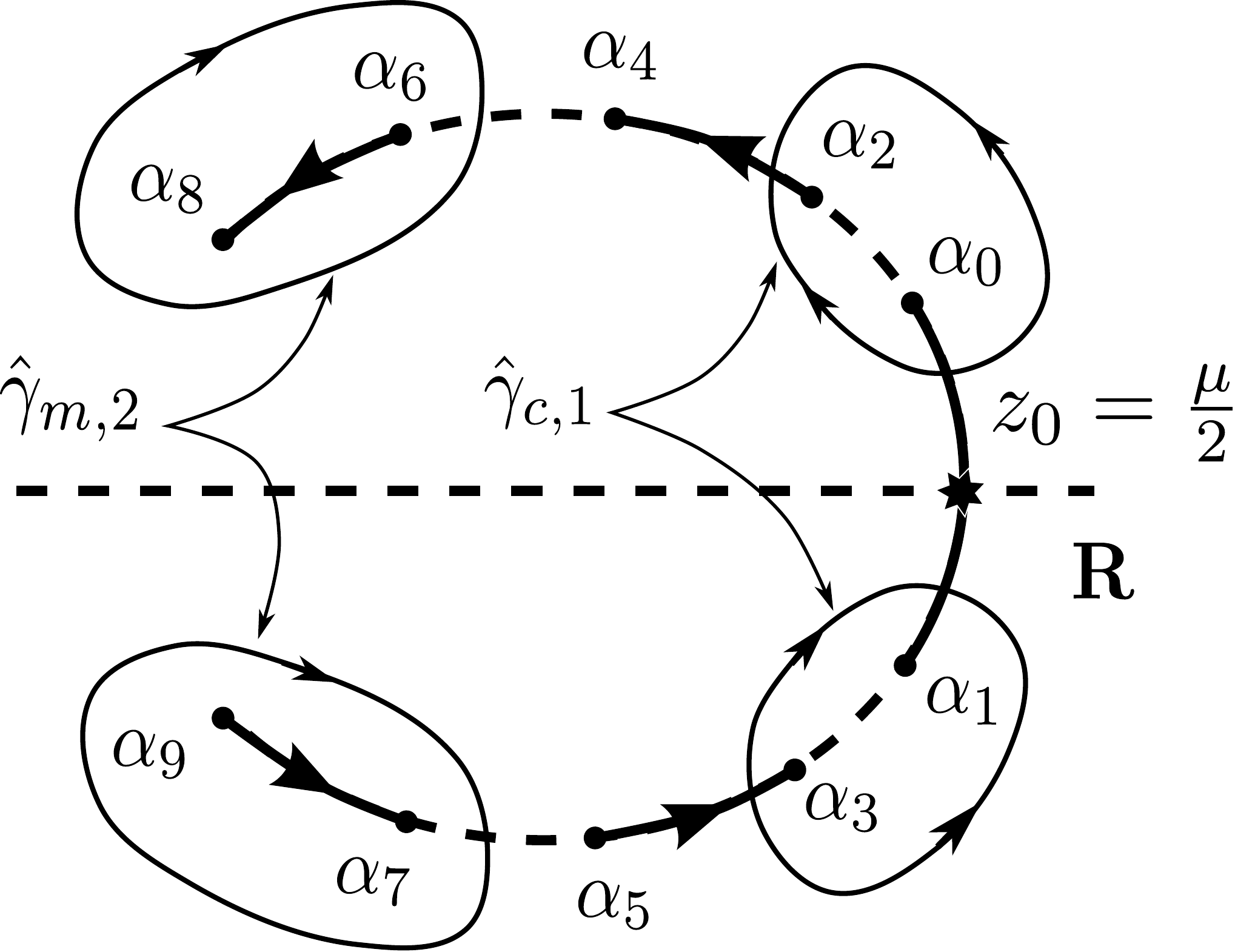}}
\caption{\label{fig:loop_contours} Contours of integration for function $h(z)$ (\ref{eq:h(z)_def_loops}). Point
$z_0=\mm$ is a point of non-analyticity of $f(z)$ on $\gamma$.}
\end{center}
\end{figure}

We consider a model scalar Riemann-Hilbert problem (RHP),  which arises  in the process of the asymptotic solution of the semiclassical focusing NLS (\ref{eq:1.1_NLS}) with the initial condition  (\ref{IC_NLS_special}).  The input to the problem is a given  function $f(z)$,
that derives originally  from the asymptotic limit of the scattering data for this initial value problem.  The function $f(z)$ (see \eqref{f}) depends parametrically on the space and time variables, $x$ and $t$. It also depends on  the real parameter $\mu$ in the initial data of NLS (\ref{IC_NLS_special}). The  following properties
of the function $f(z)$  are crucial for our calculations.
\begin{enumerate}
\item $f(z)$ is analytic at all points of $\mathbb{C}/\mathbb{R}$ except for branchcuts.
\item $f(z)$  has a $z\ln z$ singularity at the point $z=\mu/2$.
\item $f(z)$ is Schwarz-symmetric.

\end{enumerate}
Other functions  that satisfy these conditions  are a priori  admissible as inputs to our model problem; whether they too lead to solutions is a matter of calculation.

The unknown of the problem is a  function $h(z)$ that has the following properties.
\begin{enumerate}
\item The function $h(z)$ is  Schwarz-reflexive.
 \item The function $h(z)$   compensates for the points of nonanalyticity  of $f(z)$ in the sense that $h+f$ is analytic at    these points with only one exception, the point $z=\mu/2$.
 \item The function  $h+f$  is analytic in
$\mathbb{\overline C}\backslash\gamma$, where $\gamma$ is a contour
to be determines, that passes through point $z=\mu/2$, and is symmetric with respect to the real axis.
\item The function $h(z)$ exhibits  constant (independent of $z$) jumps across subarcs of the contour $\gamma$.  This is made more specific below.
\end{enumerate}
{\it Remark:} One  realizes that a cleaner formulation of the RHP could be achieved for the  unknown $f+h$ instead of $h$. Indeed, the function
\[ g(z)=\frac{1}{2}\Biggl(h(z)+f(z)\Biggr),    \]
jumps only across  contour $\gamma$, providing  the cleaner RHP formulation. Yet, the results are in terms of the function $h(z)$, hence our choice in its favor.

\noindent{\it Remark:} Based on the previous remark, we still refer to the contour $\gamma$ as the contour of the RHP. Since the contour itself is one of the  unknowns, we refer to the problem as a ``free contour RHP'', in analogy to the well-known ``free boundary value problems''.

 We now set the precise conditions for the contour   $\gamma$ and the jumps across it.
\begin{enumerate}
 \item The contour  $\gamma$  is a finite length, non-selfintersecting  arc that is  symmetric with respect to the real axis. It intersects the real axis at a point $z=\mu/2$ (we refer to this point as $z_0$) to be discussed below.
\item The contour $\gamma$ is oriented from its endpoint in the lower complex  half-plane  to its endpoint in the
upper half-plane.
\item For some integer $N$, we consider $2N+1$ distinct  points of the contour  in the upper half plane, including the contour endpoint; we also consider  their complex-conjugates   in the lower half plane. We label the points in the upper half-plane with even indices  $\{\alpha_{2i}\}_{i=0}^{2N}$ that increase in the direction of orientation of the contour and we label the points in the lower half-plane  with odd indices   $\{\alpha_{2i+1}\}_{i=0}^{2N}$ that decrease in the direction of orientation. Clearly,  the sequence of points in the the direction of orientation are
\[\underbrace{\alpha_{4N+1}, \ \alpha_{4N-1},  \  \alpha_{4N-3}, \cdots, \alpha_{3}, \  \alpha_{1}}_{\mbox{lower half-plane}}, \ \  \underbrace{\alpha_{0}, \  \alpha_{2},\cdots,  \alpha_{4N-4},  \ \alpha_{4N-2}, \  \alpha_{4N}}_{\mbox{upper half-plane}}, \]
$\alpha_{4N+1}$ and $\alpha_{4N}$ are the endpoints of the contour   $\gamma$, and
\begin{equation}
 \alpha_{2i+1}=\overline{\alpha_{2i}}, \ \ i=0, 1, 2, \cdots, 2N.
\end{equation}
\end{enumerate}

The jumps of the RHP are defined  on the arcs into which the contour is partitioned by these points.  Two alternative types of RHP jumps are imposed; each arc is labeled as a main arc or a complementary arc, respectively. The two arc types  interlace along the contour and the contour  has a main arcs at both ends.  All arcs inherit  their orientation from the contour $\gamma$.

It is trivial to check that  arcs which are complex-conjugate to each other are either both main or both complementary. It is convenient to lump such an arc pair  into one entity; in the following definitions,  we retain the terms  {\it main arc} and  {\it complementary arc} for such  arc pairs, by abuse of vocabulary.

\begin{enumerate}
 \item
We define as main arcs $\gg_{m,j}$, where $j=0, 1,...,N$,
\[
\gg_{m,0}=\left[\alpha_{1},\alpha_{0}\right],\quad
\gg_{m,j}=\left[\alpha_{4j-2},\alpha_{4j}\right]\cup\left[\alpha_{4j+1},\alpha_{4j-1}\right],\quad j=1,...,N,
\]
Thus, a main arc consists of a single arc when $j=0$ and of the union of two arcs, when $j>0$.
\item
We define as complementary arcs $\gg_{c,j}$, where $j= 1,...,N$,
\[
\gg_{c,j}=\left[\alpha_{4j-4},\alpha_{4j-2}\right]\cup\left[\alpha_{4j-1},\alpha_{4j-3}\right],\quad  j=1,...,N.
\]
\end{enumerate}

The jump conditions of the RHP are given by,
\begin{equation}
\left\{
\begin{array}{l}
h_{+}(z)+h_{-}(z)=2W_j,\ \mbox{on} \ \gg_{m,j}, \ \ j=0,1,...,N,\\
h_{+}(z)-h_{-}(z)=2\Omega_j,\ \mbox{on} \ \gg_{c,j}, \ \ j=1,...,N,\\
h(z)+f(z)\ \mbox{is analytic in}\ \mathbb{\overline{C}}\backslash\gg,
\end{array}
\right. \label{eq:RHP_h_NLS_def}
\end{equation}
where $W_j$ and $\Omega_j$ are real constants with a normalization $W_0=0$.

We end our formulation of the free contour RHP for the function $h(z)$ by reiterating what the knowns and what the unknowns of the problem are.
The contour $\gamma$ is unknown, except for the requirement of passing through the singular point $z_0=\mu/2$. The positions  of the $4N+2$ partitioning points  are unknown. As we have formulated the problem so far, the value of the integer $N$ is free.  This freedom is lifted if important  additional conditions (they are ``sign conditions'') are imposed on the RHP,  as occurs in the case of NLS \cite{TVZzero}. The sign conditions guarantee the decay of  certain jump matrix entries. In the presence of these conditions, the  number of points turns into  an important unknown, the $n$-phase wave represented has $n=2N$ .  Finally, the real constants in the jump conditions are unknown. To summarize, the only known data is the function $f(z;\mu,\beta)$.

Our main concern is the  dependence of the  solution of the problem in the parameter $\mu$. Any other parameters (space and time if the RHP arises from focusing NLS) are collectively labeled $\beta$.  A multi-parameter family of such functions $f$ will be discussed below. The specific  function $f$ that corresponds to the focusing NLS equation with initial data \eqref{IC_NLS_special} is given in the beginning of subsection $\ref{Setup}$.

\begin{definition}$ $\\
Let
\begin{equation}
 \va
=\{\a_0, \ \a_1, \ \a_2, \cdots,  \a_{4N+1} \}
\end{equation}
We say
\[
\gamma\in\Gamma(\va,\mu)
\]
if a  contour $\gamma$ satisfies all the conditions set above  and shown in Fig. \ref{fig:RHP_contours}.
\label{def:gamma}
\end{definition}

Note that for fixed $\va$ and $\mu$ the contour $\gamma$ aside
from passing through $z=\alpha_j$ and $z=\mm$ is free to deform
continuously within of domain of analyticity of $f$. Thus for a
fixed $f$, the notation  $\gamma=\gamma(\va,\mu)$ may employed to
indicate the general element of the set $\Gamma(\va,\mu)$. The
following lemma is an immediate consequence of our definitions.

\begin{lemma}
Let $\gamma_0=\gamma_0(\va_0,\mu_0)\in\Gamma(\va_0,\mu_0)$.

Then there exist open neighborhoods of $\va_0$ and $\mu_0$
such that for all $\va$ in the neighborhood of $\va_0$, for all $\mu$ in the neighborhood of $\mu_0$
there is a contour $\gamma(\va,\mu)\in\Gamma(\va,\mu).$
\label{lem:gamma_in_Gamma_near}
\end{lemma}

\begin{definition}$ $\\
We say
\[
\ggh\in\hat\Gamma(\gamma,\va,\mu)
\]
if $\ggh$ is a non-selfintersecting closed contour around
$\gamma\in\Gamma(\va,\mu)$ within the domain of analyticity of $f$
except at $z_0=\mm$, with complex-conjugate symmetry
$\overline{\ggh}=\ggh$.

\label{def:gamma_hat}
\end{definition}

$\ggh_{m,j}$ and $\ggh_{c,j}$ are defined similarly.

\begin{remark}
By considering the loop contours $\ggh$, $\ggh_{m,j}$, $\ggh_{c,j}$,
the explicit dependence of the contours on the end points $\va$ is removed
(for example in (\ref{eq:Lemma3.2_int1}-\ref{eq:Lemma3.2_int5-6})). So even though
$\gamma=\gamma(\va,\mu)$, in all our evaluations below $\ggh=\ggh(\mu)$.
\end{remark}

\begin{remark}
Lemma \ref{lem:gamma_in_Gamma_near} implies that if $\ggh_0\in\hat\Gamma(\gamma_0,\va_0,\mu_0)$ then
there is a contour $\gamma\in\Gamma(\va,\mu)$ such that
$\ggh\in\hat\Gamma(\gamma,\va,\mu)$ for all $\va$ and $\mu$ in some open neighborhoods of $\va_0$ and $\mu_0$.
\end{remark}

\begin{definition}$ $\\
We denote the RHP (\ref{eq:RHP_h_NLS_def}) as
\[
RHP(\gamma,\va,\mu,f),
\]
where $\gg\in\Gamma(\va,\mu)$.
\label{def:RHP}
\end{definition}

The solution of the RHP (\ref{eq:RHP_h_NLS_def}), $h(z)$ can be found explicitly \cite{TVZzero}
\begin{equation}
h(z)=\frac{R(z)}{2\pi i}
\left[\oint_{\ggh}\frac{f(\zeta)}{(\zeta-z)R(\zeta)}d\zeta +
\sum_{j=1}^{N}\oint_{\ggh_{m,j}}\frac{W_j}{(\zeta-z)R(\zeta)}d\zeta+
\sum_{j=1}^{N}\oint_{\ggh_{c,j}}\frac{\Omega_j}{(\zeta-z)R(\zeta)}d\zeta \right],
\label{eq:h(z)_def_loops}
\end{equation}

or in the determinant form \cite{TV_det}
\begin{equation}
h(z)=\frac{R(z)}{D}K(z),
\label{eq:h_and_K_def}
\end{equation}
where $z$ lies inside of $\ggh$ and outside all $\ggh_{c,j}$ and $\ggh_{m,j}$, and where

\begin{equation}
K(z)=\frac{1}{2\pi i}\left|\begin{array}{cccc}
\oint_{\ggh_{m,1}}\frac{d\zeta}{R(\zeta)} & \ldots & \oint_{\ggh_{m,1}}\frac{\zeta^{N-1}d\zeta}{R(\zeta)}
& \oint_{\ggh_{m,1}}\frac{d\zeta}{(\zeta-z)R(\zeta)} \\
\ldots & \ldots & \ldots & \ldots \\
\oint_{\ggh_{m,N}}\frac{d\zeta}{R(\zeta)} & \ldots & \oint_{\ggh_{m,N}}\frac{\zeta^{N-1}d\zeta}{R(\zeta)}
& \oint_{\ggh_{m,N}}\frac{d\zeta}{(\zeta-z)R(\zeta)} \\
\oint_{\ggh_{c,1}}\frac{d\zeta}{R(\zeta)} & \ldots & \oint_{\ggh_{c,1}}\frac{\zeta^{N-1}d\zeta}{R(\zeta)}
& \oint_{\ggh_{c,1}}\frac{d\zeta}{(\zeta-z)R(\zeta)} \\
\ldots & \ldots & \ldots & \ldots \\
\oint_{\ggh_{c,N}}\frac{d\zeta}{R(\zeta)} & \ldots & \oint_{\ggh_{c,N}}\frac{\zeta^{N-1}d\zeta}{R(\zeta)}
& \oint_{\ggh_{c,N}}\frac{d\zeta}{(\zeta-z)R(\zeta)} \\
\oint_{\ggh}\frac{f(\zeta)d\zeta}{R(\zeta)} & \ldots & \oint_{\ggh}\frac{\zeta^{N-1}f(\zeta)d\zeta}{R(\zeta)} &
\oint_{\ggh}\frac{f(\zeta)d\zeta}{(\zeta-z)R(\zeta)}
\end{array}\right|
\label{eq:K_function_def}
\end{equation}
and
\begin{equation}
D=det(A),
\label{eq:D_det_def}
\end{equation}
with
\begin{equation}
A=\left(\begin{array}{ccc}
\oint_{\ggh_{m,1}}\frac{d\zeta}{R(\zeta)} & \ldots & \oint_{\ggh_{m,1}}\frac{\zeta^{N-1}d\zeta}{R(\zeta)} \\
\ldots & \ldots & \ldots \\
\oint_{\ggh_{m,N}}\frac{d\zeta}{R(\zeta)} & \ldots & \oint_{\ggh_{m,N}}\frac{\zeta^{N-1}d\zeta}{R(\zeta)} \\
\oint_{\ggh_{c,1}}\frac{d\zeta}{R(\zeta)} & \ldots & \oint_{\ggh_{c,1}}\frac{\zeta^{N-1}d\zeta}{R(\zeta)} \\
\ldots & \ldots & \ldots  \\
\oint_{\ggh_{c,N}}\frac{d\zeta}{R(\zeta)} & \ldots & \oint_{\ggh_{c,N}}\frac{\zeta^{N-1}d\zeta}{R(\zeta)}
\end{array}\right).
\label{eq:A_matrix_def}
\end{equation}

The arcs end points $\left\{\alpha_j\right\}$ satisfy the system

\begin{equation}
K(\alpha_j)=0,\ \ \ j=0,1,\ldots,4N+1.
\label{eq:mod_K_def}
\end{equation}

The dependence on $x$ and $t$  was considered in \cite{TV_det}. This is a simpler situation
when the jump contour $\gg$ in the RHP (\ref{eq:RHP_h_NLS_def}) is independent of the
parameters.

The main related results in \cite{TV_det} are the determinant form (\ref{eq:h_and_K_def}) and
\begin{theorem}
Let $f(z,\vb)$, where $\vb\in B\subset\mathbb{R}^m$. For all $\vb\in B$ assume $f(z,\vb)$
be analytic in on $S\in\mathbb{C}$. Moreover, $\gamma\backslash S$ consists of no
more than finitely points and $f$ continuous on $\gamma$.
The modulation equations (\ref{eq:mod_K_def}) imply the system of $4N+2$ differential
equations
\begin{equation}
\left(\alpha_j\right)_{\beta_k}=-\frac{2\pi i\frac{\partial}{\partial\beta_k}K(\alpha_j)}{D\oint_{\ggh} \frac{f'(\zeta)}{(\zeta-\alpha_j)R(\zeta)}d\zeta}
\end{equation}
\label{thm:TV_det_result}
\end{theorem}
In particular, one gets (\ref{eq:D_alpha_D_x_intro}) and (\ref{eq:D_alpha_D_t_intro}) for parameters $x$ and $t$.
Note, that the contour $\gamma$ is assumed independent of parameters $x$ and $t$ explicitly.
The dependence on these parameters comes in through the branchpoints $\va=\va(x,t)$.

The main related result in \cite{TV_param} is
\begin{theorem}
Let the nonlinear steepest descent asymptotics for solution $q(x, t, \eps)$ of
the NLS (\ref{eq:1.1_NLS}) be valid at some point $(x_0, t_0)$. If $(x_*, t_*)$
is an arbitrary point, connected
with $(x_0, t_0)$ by a piecewise-smooth path $\Sigma$, if the contour $\gamma(x, t)$ of
the RHP (\ref{eq:RHP_h_NLS_def}) does
not interact with singularities of $f(z)$ as $(x, t)$ varies from $(x_0, t_0)$ to
$(x_*, t_*)$ along
$\Sigma$, and if all the branchpoints are bounded and stay away from the real axis, then the
nonlinear steepest descent asymptotics (with the proper choice of the genus) is also valid
at $(x_*, t_*)$.
\label{thm:TV_param_result}
\end{theorem}

We extend Theorem \ref{thm:TV_det_result} and make partial progress in the direction
of Theorem \ref{thm:TV_param_result} in the case when the jump contour explicitly depends
on the parameter $\mu$. We require that the point of logarithmic singularity
of $f=f(z,\mu)$, $z_0=\mm$ is always on $\gamma$.
Additionally we prove preservation of genus for all $x>0$, $t>0$, $\mu>0$,
under certain conditions
which guarantee that the parameters are away from asymptotic solution breaks
(see Theorem \ref{thm:pres_genus_NLS}).
In particular, the genus is preserved in a neighborhood of the special
value of the parameter $\mu=2$.
Thus we obtain that for all $x>0$, $t>0$ (except on the first breaking curve)
there is a small neighborhood such that for all $\mu<2$ in the neighborhood,
the genus is the same as for $\mu=2$, where it is known to be $0$ or $2$.

\section{$\mu$-dependence in the semiclassical focusing NLS}

\subsection{Setup\label{Setup}}

To apply the methods from \cite{TV_det} we need analyticity of $f(z,\mu)$ in
the parameter $\mu$.

The function $f(z)$ obtained in  \cite{TVZzero} from a semiclassical approximation of the exactly derived scattering data for NLS with initial condition (\ref{IC_NLS_special}) \cite{TV_ScatData} :
\begin{equation}\label{f}
f(z,\mu,x,t)=\left(\frac{\mu}{2}-z\right)\left[\frac{\pi
i}{2}+\ln\left(\frac{\mu}{2}-z\right)\right]+\frac{z+T}{2}\ln\left(z+T\right)
+\frac{z-T}{2}\ln\left(z-T\right)
\end{equation}
\begin{equation}\nonumber
-T\tanh^{-1}\frac{T}{\mu/2}-xz-2tz^2 +\frac{\mu}{2}\ln 2, \quad
\mbox{when}\ \ \Im z\ge 0 \label{eq:f-function_NLS}
\end{equation}
and
\begin{equation}
f(z)=\overline{f(\overline{z})},\quad \mbox{when}\ \ \Im z< 0,
\label{eq:Schwarz_reflection}
\end{equation}
where the branchcuts are chosen as the following: for $0<\mu<2$ the logarithmic branchcut is
 from $z=\mm$ along the real
axis to $+\infty$, from $z=T$ to $0$ and along the real axis to $+\infty$, from
$z=-T$ to $0$ and along the real axis to $-\infty$; for $\mu\ge 2$ - the branchcuts are from $z=T$ to $+\infty$
and from $z=-T$ to $-\infty$ along the real axis, where
\begin{equation}
T=T(\mu)=\sqrt{\frac{\mu^2}{4}-1},\ \ \ \Im T\ge 0.
\end{equation}
For $\mu\ge 2$, $T\ge 0$ is real and for $0<\mu<2$, $T$ is purely imaginary with $\Im T>0$.

\begin{lemma}
$f(z,\mu)$ and $f'(z,\mu)$ are analytic in $\mu$ for $\mu>0$, $x>0$, $t>0$, for all $z$, $\Im z\neq 0$, $z\notin [-T,T]$.
\label{lem:f_analytic}
\end{lemma}

\begin{proof}

Consider
\begin{equation}
f'(z,\mu)=-\frac{\pi i}{2}-\ln\left(\frac{\mu}{2}-z\right)
+\frac{1}{2}\ln\left(z^2-\frac{\mu^2}{4}+1\right)-x-4tz,
\end{equation}
which analytic in $\mu>0$, for $\Im z\neq 0$, $z\notin [-T,T]$.

For $\mu>0$, $\mu\neq 2$, $f(z,\mu)$ is clearly analytic in $\mu$ for $\Im z\neq 0$. At $\mu=2$ ($T=0$) we find the power
series of $f(z,\mu)$ in $T$ and show that it contains only even powers. Since
\begin{equation}
T^{2k}=\left(\frac{\mu^2}{4}-1\right)^{k}=\frac{(\mu+2)^k}{4^k}(\mu-2)^k
\end{equation}
it will show analyticity of $f(z,\mu)$ in $\mu$.

Start with expanding basic terms in series at $T=0$
\begin{equation}
\frac{1}{\mu/2}=\sqrt{1+T^2}^{-1}=\sum_{k=0}^\infty c_k T^{2k},\ \ \ \ln\left(z\pm T\right)=z\ln z \sum_{n=1}^{\infty}\frac{(-1)^{n+1}}{n}\left(\pm\frac{T}{z}\right)^n.
\end{equation}
Then the logarithmic terms in (\ref{eq:f-function_NLS}) become
\begin{equation}
\frac{z+T}{2}\ln\left(z+T\right)+\frac{z-T}{2}\ln\left(z-T\right)
\end{equation}
\begin{equation}
=z\ln z - z \sum_{n\ is\ even} \frac{1}{n}\left(\frac{T}{z}\right)^n + T\sum_{n\ is\ odd}\frac{1}{n}\left(\frac{T}{z}\right)^n
\end{equation}
\begin{equation}
=z\ln z + \sum_{k=1}^{\infty} \frac{1}{2k(2k-1)z^{2k-1}} T^{2k},
\end{equation}
which has only even powers of $T$ and is analytic in $\mu$ for $\Im z\neq 0$.
Next we consider the inverse hyperbolic tangent term in (\ref{eq:f-function_NLS}) and
taking into account that $\tanh^{-1}z$ is an odd function
\begin{equation}
T\tanh^{-1}\frac{T}{\mu/2}=T\tanh^{-1}\frac{T}{\sqrt{1+T^2}}
=T\tanh^{-1}\sum_{k=0}^\infty c_k T^{2k+1}
\end{equation}
\begin{equation}
=T\sum_{k=0}^\infty \tilde{c}_k T^{2k+1}=\sum_{k=0}^\infty \tilde{c}_k T^{2k+2},
\end{equation}
which also has only even powers of $T$.

So $f(z,x,t,\mu)$ is analytic in $\mu$ for $\mu>0$, $x>0$, $t>0$, $\Im z\neq 0$, $z\notin [-T(\mu),T(\mu)]$.

\end{proof}

Then for $\Im z> 0$
\begin{equation}
\frac{\partial f}{\partial\mu}(z,\mu)=\frac{\pi i}{4}+\frac{1}{2}\ln\left(\frac{\mu}{2}-z\right)+\ln 2+ \frac{\mu}{8T}\left[\ln(z+T)-\ln(z-T)-2\tanh^{-1}\frac{2T}{\mu}\right]
\label{eq:f'_mu}
\end{equation}

where $\tanh^{-1} x = x+O(x^3)$, as $x\to 0$ then

\begin{equation}
\frac{\partial f}{\partial\mu}(z,\mu) = \frac{\pi i}{4}+\frac{1}{2}\ln\left(\frac{\mu}{2}-z\right)+\ln 2+ \frac{\mu}{4z} -\frac{1}{2}+O(T),\ T\to 0.
\end{equation}

So $\mu=2$ is a removable singularity for $f_\mu(z,\mu)$ and
\begin{equation}
\lim_{\mu\to 2, T\to 0} \frac{\partial f}{\partial\mu}(z,\mu)=\frac{\pi i}{4}+\frac{1}{2}\ln\left(1-z\right)+\ln 2 + \frac{1}{2z} -\frac{1}{2},
\end{equation}
which is analytic in $z$ for $\Im z\neq 0$.

Note the jump of $f(z)$ is caused by the Schwarz reflection (\ref{eq:Schwarz_reflection}) on the real axis and
it is linear in $z$ since $\Im f$ is a linear function on the real axis (as a limit)
near $\mm$ with $\Im f \left(\mm\right)=0$ \cite{TVZzero}.

\subsection{Parametric dependence of scalar RHP}

The main difficulty is the dependence of $f(z)$ (thus the RHP (\ref{eq:RHP_h_NLS_def}))
and the modulation
equations (\ref{eq:mod_K_def}) on parameter $\mu$ which also controls the logarithmic
branchpoint $z=\mm$ on the contour $\ggh$. We show that the dependence on $\mu$ is smooth.

To solve $\vec{K}(\va,\mu)=\vec{0}$, we need nondegeneracy of the system and apply the implicit function theorem.
The following technical lemma simplifies expressions for partial derivatives in $\mu$
of (\ref{eq:h(z)_def_loops}) and (\ref{eq:K_function_def}).

\begin{lemma}$ $\\

Let function $f$ be given by (\ref{eq:f-function_NLS}) and there is a contour $\gamma_0=\gamma(\va,\mu_0)\in\Gamma(\va,\mu_0)$
which has fixed arcs end points $\va$.
Then there is an open neighborhood of $\mu_0$ such that for all $\mu$ in the neighborhood of $\mu_0$
there is $\ggh(\mu)\in\hat\Gamma(\gamma,\va,\mu)$ and for all $j=0,1,\ldots,4N+1$, $n\in\mathbb{N}$

\begin{equation}
\frac{\partial }{\partial\mu}\oint_{\ggh(\mu)}\frac{\zeta^n f(\zeta,\mu)d\zeta}{R(\zeta,\va)}
=\oint_{\ggh(\mu)}\frac{\zeta^n \frac{\partial f(\zeta,\mu)}{\partial\mu}d\zeta}{R(\zeta,\va)},%\ \ n\in\N,
\label{eq:Lemma3.2_int1}
\end{equation}

\begin{equation}
\frac{\partial }{\partial\mu} \oint_{\ggh(\mu)}\frac{f(\zeta,\mu)d\zeta}{(\zeta-\alpha_j)R(\zeta,\va)}
=\oint_{\ggh(\mu)}\frac{\frac{\partial f(\zeta,\mu)}{\partial\mu}d\zeta}{(\zeta-\alpha_j)R(\zeta,\va)},
\label{eq:Lemma3.2_int2}
\end{equation}

\begin{equation}
\frac{\partial }{\partial\mu} \oint_{\ggh_{m,k}}\frac{\zeta^n d\zeta}{R(\zeta,\va)}=0, \ \ \
\frac{\partial }{\partial\mu} \oint_{\ggh_{m,k}}\frac{d\zeta}{(\zeta-\alpha_j)R(\zeta,\va)}=0, \ \ k=1,2,\ldots,N,%\ n\in\N,
\label{eq:Lemma3.2_int3-4}
\end{equation}

\begin{equation}
\frac{\partial }{\partial\mu} \oint_{\ggh_{c,k}}\frac{\zeta^n d\zeta}{R(\zeta,\va)}=0, \ \ \
\frac{\partial }{\partial\mu} \oint_{\ggh_{c,k}}\frac{d\zeta}{(\zeta-\alpha_j)R(\zeta,\va)}=0, \ \ k=1,2,\ldots,N.%\ n\in\N.
\label{eq:Lemma3.2_int5-6}
\end{equation}

\label{lem:par_int_zero_NLS}

\end{lemma}

\begin{proof}$ $\\

The idea of the proof is to consider finite differences and take the limit as $\Delta\mu\to 0$.
The complication is that both the integrands and the contours of integration depend on $\mu$.

Denote the integral on the left in (\ref{eq:Lemma3.2_int1}) as $I_1$
\begin{equation}
I_1(\mu)=\oint_{\ggh(\mu)}\frac{\zeta^n f(\zeta,\mu)}{R(\zeta,\va)}d\zeta.
\end{equation}
where $\ggh(\mu)\in\hat\Gamma(\gamma,\va,\mu)$.
Consider
\[
\frac{I_1(\mu+\Delta\mu)-I_1(\mu)}{\Delta \mu},
\]
with small real $\Delta\mu\neq 0$. There are two logarithmic branchcuts near the contours of integration:
in $f(z,\mu)$ and in $f(z,\mu+\Delta\mu)$ with both branchcuts are chosen from $z_0(\mu)=\mm$
and $z_0(\mu+\Delta\mu)$ horizontally to the right along the real axis. Additionally,
these functions have a jump on the real axis for $z<\mm$ from Schwarz symmetry.

We choose some fixed points $\delta_1$ and $\delta_2$ to be real, $\delta_1<\mm-\frac{|\Delta\mu|}{2}<\mm+\frac{|\Delta\mu|}{2}<\delta_2$.
Both contours of integration $\ggh(\mu)$, $\ggh(\mu+\Delta\mu)$ are pushed to the real axis near
$z_0$ and split into $\left[\delta_1,\delta_2\right]:=\left[\delta_1+i0,\delta_2+i0\right]\cup\left[\delta_2-i0,\delta_1-i0\right]$
and its complement. On the complement, we can also deform both contours to coincide. So $\ggh(\mu+\Delta\mu)=\ggh(\mu)$.

Note that across $\left[\delta_1,\delta_2\right]$, $f(z,\mu)$ has a jump $\pi i \left|z_0(\mu)-z\right|$ and
$f(z,\mu+\Delta\mu)$ has a jump $\pi i \left|z_0(\mu+\Delta\mu)-z\right|$. So contributions near $z_0$ in both
cases are small.

Then
\begin{equation}
\frac{I_1(\mu+\Delta\mu)-I_1(\mu)}{\Delta \mu}
=\frac{1}{\Delta\mu}\left(\oint_{\ggh(\mu+\Delta\mu)}\frac{\zeta^n f(\zeta,\mu+\Delta\mu)}{R(\zeta,\va)}d\zeta
-\oint_{\ggh(\mu)}\frac{\zeta^n f(\zeta,\mu)}{R(\zeta,\va)}d\zeta\right)\\
\end{equation}
we add and subtract $\oint_{\ggh(\mu)}\frac{\zeta^n f(\zeta,\mu+\Delta\mu)}{R(\zeta,\va)}d\zeta$
\begin{equation}
=\frac{1}{\Delta\mu}\left(\oint_{\ggh(\mu+\Delta\mu)-\ggh(\mu)}\frac{\zeta^n f(\zeta,\mu+\Delta\mu)}{R(\zeta,\va)}d\zeta
+\oint_{\ggh(\mu)}\frac{\zeta^n (f(\zeta,\mu+\Delta\mu)-f(\zeta,\mu))}{R(\zeta,\va)}d\zeta\right)
\end{equation}
The first integral is 0, because $\ggh(\mu+\Delta\mu)=\ggh(\mu)$.

Thus
\begin{equation}
\frac{I_1(\mu+\Delta\mu)-I_1(\mu)}{\Delta \mu}
=\oint_{\ggh(\mu)}\frac{\zeta^n \frac{f(\zeta,\mu+\Delta\mu)-f(\zeta,\mu)}{\Delta \mu}}{R(\zeta,\va)}d\zeta .
\end{equation}

The last step is to take the limit as $\Delta\mu\to 0$ and to interchange it with the integral. The contour
of integration is split into two:
a small neighborhood near $z_0$ and its complement. For the integral near $z_0$, by a direct computation
can be shown that the limit can be passed under the integral. The integral over the second part of the contour
has the integrand uniformly bounded
in $\mu$ since $\log\left(\zeta-\mm\right)$ in $\frac{\partial f}{\partial\mu}$
is uniformly bounded away from $\mm$, so the limit and the integral can be interchanged.
This completes the proof for the first integral (\ref{eq:Lemma3.2_int1}).

The second integral (\ref{eq:Lemma3.2_int2}) is done similarly. The rest of the integrals (\ref{eq:Lemma3.2_int3-4})-(\ref{eq:Lemma3.2_int5-6})
are independent of $\mu$ since the only dependence on $\mu$ sits in $z_0(\mu)\in\gamma_{m,0}$.

\end{proof}

Using Lemma \ref{lem:par_int_zero_NLS},
\begin{equation}
\frac{\partial K}{\partial\mu}(\alpha_j,\va,\mu)=
\frac{1}{2\pi i}\left|\begin{array}{cccc}
\oint_{\ggh_{m,1}}\frac{d\zeta}{R(\zeta)} & \ldots & \oint_{\ggh_{m,1}}\frac{\zeta^{N-1}d\zeta}{R(\zeta)}
& \oint_{\ggh_{m,1}}\frac{d\zeta}{(\zeta-\alpha_j)R(\zeta)} \\
\ldots & \ldots & \ldots & \ldots \\
\oint_{\ggh_{m,N}}\frac{d\zeta}{R(\zeta)} & \ldots & \oint_{\ggh_{m,N}}\frac{\zeta^{N-1}d\zeta}{R(\zeta)}
& \oint_{\ggh_{m,N}}\frac{d\zeta}{(\zeta-\alpha_j)R(\zeta)} \\
\oint_{\ggh_{c,1}}\frac{d\zeta}{R(\zeta)} & \ldots & \oint_{\ggh_{c,1}}\frac{\zeta^{N-1}d\zeta}{R(\zeta)}
& \oint_{\ggh_{c,1}}\frac{d\zeta}{(\zeta-\alpha_j)R(\zeta)} \\
\ldots & \ldots & \ldots & \ldots \\
\oint_{\ggh_{c,N}}\frac{d\zeta}{R(\zeta)} & \ldots & \oint_{\ggh_{c,N}}\frac{\zeta^{N-1}d\zeta}{R(\zeta)}
& \oint_{\ggh_{c,N}}\frac{d\zeta}{(\zeta-\alpha_j)R(\zeta)} \\
\oint_{\ggh}\frac{f_\mu(\zeta)d\zeta}{R(\zeta)} & \ldots & \oint_{\ggh}\frac{\zeta^{N-1}f_\mu(\zeta)d\zeta}{R(\zeta)} &
\oint_{\ggh}\frac{f_\mu(\zeta)d\zeta}{(\zeta-\alpha_j)R(\zeta)}
\end{array}\right|,
\label{eq:DK/Dmu_function_NLS}
\end{equation}
where $f_\mu$ is given by (\ref{eq:f'_mu}).

\begin{lemma}$ $\\
Let $f$ be given by (\ref{eq:f-function_NLS}) and a contour $\gamma_0=\gamma(\va_0,\mu_0)\in\Gamma(\va_0,\mu_0)$
has the arcs end points $\va_0$. Then there are open neighborhoods of $\va_0$ and $\mu_0$
such that for all $\va$ and $\mu$ in the neighborhoods of $\va_0$ and $\mu_0$
respectively there is a contour $\gamma=\gamma(\va,\mu)\in\Gamma(\va,\mu)$ and
\begin{equation}
K_j(\va,\mu):=K(\alpha_j,\va,\mu),\ \ \ \ j=0,1,\ldots, 4N+1
\end{equation}
is continuously differentiable in $\va$ and in $\mu$.
\label{lem:smooth_F_mu}
\end{lemma}

\begin{proof}$ $\\
First we notice that since $\gamma_0\in\Gamma(\va_0,\mu_0)$ by Lemma \ref{lem:gamma_in_Gamma_near},
there are neighborhoods of $\va_0$ and $\mu_0$
such that for all $\va$ and $\mu$ in the neighborhoods of $\va_0$ and $\mu_0$
respectively there is a contour $\gamma=\gamma(\va,\mu)\in\Gamma(\va,\mu)$.

${K}_j(\va,\mu)$ is analytic in $\va$ by the determinant structure and the integral entries (\ref{eq:K_function_def}),
where explicit dependence on $\va$ is only in the $R(z,\va)$ term which is analytic away from $z=\alpha_j$, $j=0,...,4N+1$.

The integrals in the last row of $\frac{\partial K}{\partial\mu}(\alpha_j,\va,\mu)$ in (\ref{eq:DK/Dmu_function_NLS})
involve the function $f_\mu$ given by (\ref{eq:f'_mu}) which is integrable near $z=\mm$ and \
hence $\frac{\partial K}{\partial\mu}(\alpha_j,\va,\mu)$ is continuous in $\mu$. Thus $K_j(\alpha_j,\va,\mu)$
is continuously differentiable in $\va$ and in $\mu$.

\end{proof}

By Lemma \ref{lem:smooth_F_mu} the modulation equations (\ref{eq:mod_K_def})
\begin{equation}
K_j(\va,\mu)=K(\alpha_j,\va,\mu)=0
\label{eq:modulation}
\end{equation}
are smooth in $\va$ and in the parameter $\mu$. Next we want to solve this system for $\va=\va(\mu)$
and conclude smoothness in $\mu$ by the implicit function theorem.

For the next lemma we need $K'(z,\va,\mu)=\frac{dK}{dz}(z,\va,\mu)$
\begin{equation}
K'(z,\va,\mu)=
\frac{1}{2\pi i}\left|\begin{array}{ccccccc}
\oint_{\ggh_{m,1}}\frac{d\zeta}{R(\zeta)} & \ldots & \oint_{\ggh_{m,1}}\frac{\zeta^{N-1}d\zeta}{R(\zeta)} &
\oint_{\ggh_{m,1}}\frac{d\zeta}{(\zeta-z)^2R(\zeta)} \\
\ldots & \ldots & \ldots & \ldots  \\
\oint_{\ggh_{m,N}}\frac{d\zeta}{R(\zeta)} & \ldots & \oint_{\ggh_{m,N}}\frac{\zeta^{N-1}d\zeta}{R(\zeta)} &
\oint_{\ggh_{m,N}}\frac{d\zeta}{(\zeta-z)^2R(\zeta)} \\
\oint_{\ggh_{c,1}}\frac{d\zeta}{R(\zeta)} & \ldots & \oint_{\ggh_{c,1}}\frac{\zeta^{N-1}d\zeta}{R(\zeta)} &
\oint_{\ggh_{c,1}}\frac{d\zeta}{(\zeta-z)^2R(\zeta)} \\
\ldots & \ldots & \ldots & \ldots  \\
\oint_{\ggh_{c,N}}\frac{d\zeta}{R(\zeta)} & \ldots & \oint_{\ggh_{c,N}}\frac{\zeta^{N-1}d\zeta}{R(\zeta)} &
\oint_{\ggh_{c,N}}\frac{d\zeta}{(\zeta-z)^2R(\zeta)} \\
\oint_{\ggh}\frac{f(\zeta)d\zeta}{R(\zeta)} & \ldots & \oint_{\ggh}\frac{\zeta^{N-1}f(\zeta)d\zeta}{R(\zeta)} &
\oint_{\ggh}\frac{f(\zeta)d\zeta}{(\zeta-z)^2R(\zeta)}
\end{array}\right|,
\label{eq:K'_function}
\end{equation}
where $z$ is inside of $\ggh(\mu)$ and inside of $\ggh_{m,j}$ and $\ggh_{c,j}$ or $\ggh_{c,j+1}$.

\begin{lemma}$ $\\
Let $f$ be given by (\ref{eq:f-function_NLS}) and a contour $\gamma_0\in\Gamma(\va_0,\mu_0)$,
where $\va_0$ and $\mu_0$ satisfy
\[
\vec{K}(\va_0,\mu_0)=\vec{0}.
\]
Assume that for $\va_0=\left\{\alpha^{0}_j\right\}_{j=0}^{4N+1}$,
$\lim_{z\to\alpha^0_j}K'(z,\va_0,\mu_0)\neq 0,$ $j=0, 1,\ldots, 4N+1$.

Then the modulation equations
\[
\vec{K}(\va,\mu)=\vec{0}
\]
can be uniquely solved for $\va=\va(\mu)$ which is continuously
differentiable for all $\mu$ in some open neighborhood of $\mu_0$ and $\va(\mu_0)=\va_0$.
\label{lem:smooth_alpha_mu}
\end{lemma}

\begin{proof}

$\vec{K}$ is continuously differentiable in $\va$ and in $\mu$ by Lemma \ref{lem:smooth_F_mu}.

As it was shown in \cite{TV_det}, the matrix
$\left\{\frac{\partial \vec{K}}{\partial\va}\right\}_{j,l}
=\left\{\frac{\partial K(\alpha_j)}{\partial\alpha_l}\right\}_{j,l}$
is diagonal and
\begin{equation}
\frac{\partial K(\alpha_j)}{\partial\alpha_j}=\frac{3}{2}D\lim_{z\to\alpha^0_j}\left(\frac{h(z)}{R(z)}\right)'=
\frac{3}{2}\lim_{z\to\alpha^0_j}K'(z,\va,\mu) \neq 0.
\end{equation}
So
\begin{equation}
\det\left|\frac{\partial \vec{K}}{\partial\vec\alpha}(\va_0)\right|=\prod_{j}\frac{\partial K(\alpha_j)}{\partial\alpha_j}\neq 0
\end{equation}
under the assumptions. By the implicit function theorem, $\va(\mu)$
are uniquely defined in some neighborhood of $\mu_0$ and smooth in $\mu$.
Note $\va(\mu_0)=\va_0$ by assumption.
\end{proof}

\begin{remark}$ $\\
The condition $\lim_{z\to\alpha^{0}_j}K'(z,\va_0,\mu_0)\neq 0,$ $j=0, 1,\ldots, 4N+1$ in Lemma
\ref{lem:smooth_alpha_mu} is equivalent to $\lim_{z\to\alpha^{0}_j}\frac{h'(z,\va_0,\mu_0)}{R(z,\va_0)}\neq 0,$ $j=0, 1,\ldots, 4N+1$.

\end{remark}

All quantities below depend on parameters $x$ and $t$. We assume that for the rest of the paper
$x$ and $t$ are fixed.

\begin{theorem} ($\mu$-perturbation in genus $N$)\\
Consider a finite length non-self-intersecting contour $\gamma_0$ in the complex plane consisting of a finite union of oriented arcs  $\gamma_0=\left(\cup\gg_{m,j}\right)\cup\left(\cup\gg_{c,j}\right)\in\Gamma(\va_0,\mu_0)$
 with the distinct arcs end points $\va_0$ and depending on parameter $\mu$ (see Figure \ref{fig:RHP_contours}). Assume $\va_0$
and $\mu_0$ satisfy a system of equations
\[
\vec{K}\left(\va_0,\mu_0\right)=\vec{0},
\label{eq:K_0}
\]
and $f$ is given by (\ref{eq:f-function_NLS}).
Let $\gamma=\gamma(\va,\mu)$ be a contour of a RH problem which seeks a function
$h(z)$ which satisfies the following conditions

\begin{equation}
\left\{
\begin{array}{l}
h_{+}(z)+h_{-}(z)=2W_j,\ \mbox{on} \ \gg_{m,j}, \ \ j=0,1,...,N,\\
h_{+}(z)-h_{-}(z)=2\Omega_j,\ \mbox{on} \ \gg_{c,j}, \ \ j=1,...,N,\\
h(z)+f(z)\ \mbox{is analytic in}\ \mathbb{\overline{C}}\backslash\gg,
\end{array}
\right. \label{eq:RHP_h_NLS}
\end{equation}
where $\Omega_j=\Omega_j(\va,\mu)$ and $W_j=W_j(\va,\mu)$ are real constants (with normalization $W_0=0$) whose numerical values will be determined from the RH conditions.
Assume that there is a function $h(z,\va_0,\mu_0)$ which satisfies (\ref{eq:RHP_h_NLS}) and suppose $\frac{h'(z,\va_0,\mu_0)}{R(z,\va_0)}\neq 0$
for all $z$ on $\gamma_0$.

Then there is a contour $\gg(\va,\mu)\in\Gamma(\va,\mu)$ such that the solution $\va=\va(\mu)$ of the system
\begin{equation}
\vec{K}\left(\va,\mu\right)=\vec{0}
\label{eq:vec_K_system_NLS}
\end{equation}
and $h(z,\va(\mu),\mu)$ which solves (\ref{eq:RHP_h_NLS}) are uniquely defined and continuously differentiable in $\mu$ in some
open neighborhood of $\mu_0$.

Moreover,
\begin{equation}
\frac{\partial\alpha_j}{\partial\mu}(\mu)=-\frac{2\pi i\ \frac{\partial K}{\partial\mu}(\alpha_j(\mu),\va(\mu),\mu)}{D(\va(\mu),\mu) \oint_{\ggh(\mu)}\frac{f'(\zeta,\mu)}{(\zeta-\alpha_j(\mu))R(\zeta,\va(\mu))}d\zeta},
\label{eq:alpha_mu_thm}
\end{equation}
\begin{equation}
\frac{\partial h}{\partial\mu}(z,\mu)=\frac{R(z,\va(\mu))}{2\pi i}\oint_{\ggh(\mu)}
\frac{\frac{\partial f}{\partial\mu}(\zeta,\mu)}{(\zeta-z)R(\zeta,\va(\mu))}d\zeta,
\label{eq:h_mu_thm}
\end{equation}
where $z$ is inside of $\ggh$.

Furthermore,
$\Omega_j(\mu)=\Omega_j(\va(\mu),\mu)$, and $W_j(\mu)=W_j(\va(\mu),\mu)$
are defined and continuously differentiable in $\mu$ in some open neighborhood of $\mu_0$, and
\begin{equation}
\begin{array}{cc}
\frac{\partial\Omega_j}{\partial\mu}(\mu)=\frac{-1}{D}
\left|\begin{array}{ccc}
\oint_{\ggh_{m,1}}\frac{d\zeta}{R(\zeta)} & \ldots & \oint_{\ggh_{m,1}}\frac{\zeta^{N-1}d\zeta}{R(\zeta)}  \\
\ldots & \ldots & \ldots  \\
\oint_{\ggh_{m,N}}\frac{d\zeta}{R(\zeta)} & \ldots & \oint_{\ggh_{m,N}}\frac{\zeta^{N-1}d\zeta}{R(\zeta)}   \\
\oint_{\ggh_{c,1}}\frac{d\zeta}{R(\zeta)} & \ldots & \oint_{\ggh_{c,1}}\frac{\zeta^{N-1}d\zeta}{R(\zeta)}  \\
\ldots & \ldots & \ldots  \\
\oint_{\ggh_{c,j-1}}\frac{d\zeta}{R(\zeta)} & \ldots & \oint_{\ggh_{c,j-1}}\frac{\zeta^{N-1}d\zeta}{R(\zeta)}  \\
{\oint_{\ggh}\frac{ f_{\mu}(\zeta)}{R(\zeta,\va)}d\zeta} & \ldots & {\oint_{\ggh}\frac{\zeta^{N-1} f_{\mu}(\zeta)}{R(\zeta,\va)}d\zeta} \\
\oint_{\ggh_{c,j+1}}\frac{d\zeta}{R(\zeta)} & \ldots & \oint_{\ggh_{c,j+1}}\frac{\zeta^{N-1}d\zeta}{R(\zeta)}  \\
\ldots & \ldots & \ldots \\
\oint_{\ggh_{c,N}}\frac{d\zeta}{R(\zeta)} & \ldots & \oint_{\ggh_{c,N}}\frac{\zeta^{N-1}d\zeta}{R(\zeta)}
\end{array}\right|, &
\frac{\partial W_j}{\partial\mu}(\mu)=\frac{-1}{D}
\left|\begin{array}{cccccc}
\oint_{\ggh_{m,1}}\frac{d\zeta}{R(\zeta)} & \ldots & \oint_{\ggh_{m,1}}\frac{\zeta^{N-1}d\zeta}{R(\zeta)}  \\
\ldots & \ldots & \ldots \\
\oint_{\ggh_{m,j-1}}\frac{d\zeta}{R(\zeta)} & \ldots & \oint_{\ggh_{m,j-1}}\frac{\zeta^{N-1}d\zeta}{R(\zeta)}  \\
{\oint_{\ggh}\frac{ f_{\mu}(\zeta)}{R(\zeta,\va)}d\zeta} & \ldots & {\oint_{\ggh}\frac{\zeta^{N-1} f_{\mu}(\zeta)}{R(\zeta,\va)}d\zeta} \\
\oint_{\ggh_{m,j+1}}\frac{d\zeta}{R(\zeta)} & \ldots & \oint_{\ggh_{m,j+1}}\frac{\zeta^{N-1}d\zeta}{R(\zeta)}  \\
\ldots & \ldots & \ldots  \\
\oint_{\ggh_{m,N}}\frac{d\zeta}{R(\zeta)} & \ldots & \oint_{\ggh_{m,N}}\frac{\zeta^{N-1}d\zeta}{R(\zeta)}  \\
\oint_{\ggh_{c,1}}\frac{d\zeta}{R(\zeta)} & \ldots & \oint_{\ggh_{c,1}}\frac{\zeta^{N-1}d\zeta}{R(\zeta)}  \\
\ldots & \ldots & \ldots  \\
\oint_{\ggh_{c,N}}\frac{d\zeta}{R(\zeta)} & \ldots & \oint_{\ggh_{c,N}}\frac{\zeta^{N-1}d\zeta}{R(\zeta)}
\end{array}\right|,
\end{array}
\label{eq:Omega_mu_thm}
\end{equation}
%\begin{equation}
%\frac{\partial W_j}{\partial\mu}(\mu)=\frac{-1}{D}
%\left|\begin{array}{cccccc}
%\oint_{\ggh_{m,1}}\frac{d\zeta}{R(\zeta)} & \ldots & \oint_{\ggh_{m,1}}\frac{\zeta^{N-1}d\zeta}{R(\zeta)}  \\
%\ldots & \ldots & \ldots \\
%\oint_{\ggh_{m,j-1}}\frac{d\zeta}{R(\zeta)} & \ldots & \oint_{\ggh_{m,j-1}}\frac{\zeta^{N-1}d\zeta}{R(\zeta)}  \\
%{\oint_{\ggh}\frac{ f_{\mu}(\zeta)}{R(\zeta,\va)}d\zeta} & \ldots & {\oint_{\ggh}\frac{\zeta^{N-1} f_{\mu}(\zeta)}{R(\zeta,\va)}d\zeta} \\
%\oint_{\ggh_{m,j+1}}\frac{d\zeta}{R(\zeta)} & \ldots & \oint_{\ggh_{m,j+1}}\frac{\zeta^{N-1}d\zeta}{R(\zeta)}  \\
%\ldots & \ldots & \ldots  \\
%\oint_{\ggh_{m,N}}\frac{d\zeta}{R(\zeta)} & \ldots & \oint_{\ggh_{m,N}}\frac{\zeta^{N-1}d\zeta}{R(\zeta)}  \\
%\oint_{\ggh_{c,1}}\frac{d\zeta}{R(\zeta)} & \ldots & \oint_{\ggh_{c,1}}\frac{\zeta^{N-1}d\zeta}{R(\zeta)}  \\
%\ldots & \ldots & \ldots  \\
%\oint_{\ggh_{c,N}}\frac{d\zeta}{R(\zeta)} & \ldots & \oint_{\ggh_{c,N}}\frac{\zeta^{N-1}d\zeta}{R(\zeta)}
%\end{array}\right|,
%\label{eq:W_mu_thm}
%\end{equation}
where $R(\zeta)=R(\zeta,\va)$, $f(\zeta)=f(\zeta,\mu)$, $f_\mu(\zeta)=\frac{\partial f}{\partial\mu}(\zeta,\mu)$, and
$D=D(\va(\mu))$.
\label{thm:perturb_NLS}
\end{theorem}

\begin{proof}

By Lemma \ref{lem:smooth_alpha_mu} there is a contour $\gg(\va,\mu)\in\Gamma(\va,\mu)$ for all $\mu$ in some neighborhood of $\mu_0$ and $\alpha_j(\mu)$ are continuously differentiable in $\mu$.
Formula for $\frac{\partial\alpha_j}{\partial\mu}$ is derived similarly as in \cite{TV_det}.
We differentiate the modulation equations $K(\alpha_j)=K(\alpha_j,\va,\mu)=0$
which define $\va=\va(\mu)$ with respect to $\mu$
\begin{equation}
\sum_{l=0}^{4N+1}\frac{\partial K(\alpha_j)}{\partial\alpha_l}\frac{\partial\alpha_l}{\partial\mu}+\frac{\partial K}{\partial\mu}(\alpha_j)=0,
\end{equation}
where the matrix $\left\{\frac{\partial K(\alpha_j)}{\partial\alpha_l}\right\}_{j,l}$ is diagonal \cite{TV_det} so
\begin{equation}
\frac{\partial K(\alpha_j)}{\partial\alpha_j}\frac{\partial\alpha_j}{\partial\mu}=-\frac{\partial K(\alpha_j)}{\partial\mu}.
\end{equation}
Since
\begin{equation}
\frac{\partial K(\alpha_j)}{\partial\alpha_j}=\frac{D(\va,\mu)}{2\pi i}\oint_{\ggh(\mu)}\frac{f'(\zeta,\mu)}{(\zeta-\alpha_j)R(\zeta,\va)}d\zeta
\end{equation}
we arrive to the evolution equations for $\alpha_j$:
\begin{equation}
\frac{\partial\alpha_j}{\partial\mu}=-\frac{2\pi i \frac{\partial K}{\partial\mu}(\alpha_j)}{D(\va,\mu) \oint_{\ggh(\mu)}\frac{f'(\zeta,\mu)}{(\zeta-\alpha_j)R(\zeta,\va)}d\zeta},\ \ j=0,...,4N+1.
\label{eq_alpha_mu}
\end{equation}

Next we compute $\frac{\partial h}{\partial\mu}$ which satisfies the scalar RHP
\begin{equation}
\left\{
\begin{array}{l}
h_{\mu,+}(z)+h_{\mu,-}(z)=0,\ \ \ z\in \gamma_{m,j}, \ j=0,1,...,N,\\
h_{\mu}(z)+f_{\mu}(z)\ \mbox{is analytic in}\ \mathbb{\overline{C}}\backslash\gg,
\end{array}
\right.
\end{equation}
Then
\begin{equation}
\frac{\partial h}{\partial\mu}(z,\mu)=\frac{R(z,\va(\mu))}{2\pi i}\oint_{\ggh(\mu)}
\frac{\frac{\partial f}{\partial\mu}(\zeta,\mu)}{(\zeta-z)R(\zeta,\va(\mu))}d\zeta,
\end{equation}
where $z$ is inside of $\ggh$. The integrand $\frac{\partial f}{\partial\mu}(\zeta,\mu)$
behaves like $\log(\zeta-z_0)$ near $\zeta=z_0$, which is integrable.

Constants $W_j$ and $\Omega_j$ are found from the linear system \cite{TVZzero}
\begin{equation}
\oint_{\ggh(\mu)}\frac{\zeta^n f(\zeta,\mu)}{R(\zeta,\va)}d\zeta +
\sum_{j=1}^{N}\oint_{\ggh_{c,j}}\frac{\zeta^n \Omega_j}{R(\zeta,\va)}d\zeta+
\sum_{j=1}^{N}\oint_{\ggh_{m,j}}\frac{\zeta^n W_j}{R(\zeta,\va)}d\zeta =0,\ n=0,...N-1.
\end{equation}

Differentiating in $\mu$ and using Lemma \ref{lem:par_int_zero_NLS} lead to

\begin{equation}
\oint_{\ggh(\mu)}\frac{\zeta^n f_{\mu}(\zeta,\mu)}{R(\zeta,\va)}d\zeta +
\sum_{j=1}^{N}\oint_{\ggh_{c,j}}\frac{\zeta^n (\Omega_{j})_{\mu}}{R(\zeta,\va)}d\zeta+
\sum_{j=1}^{N}\oint_{\ggh_{m,j}}\frac{\zeta^n (W_{j})_{\mu}}{R(\zeta,\va)}d\zeta =0,\ n=0,...,N-1
\end{equation}
or in matrix form
\begin{equation}
\left(\begin{array}{cccccc}
\oint_{\ggh_{m,1}}\frac{d\zeta}{R(\zeta)} & \ldots & \oint_{\ggh_{m,1}}\frac{\zeta^{N-1}d\zeta}{R(\zeta)} \\
\ldots & \ldots & \ldots  \\
\oint_{\ggh_{m,N}}\frac{d\zeta}{R(\zeta)} & \ldots & \oint_{\ggh_{m,N}}\frac{\zeta^{N-1}d\zeta}{R(\zeta)}  \\
\oint_{\ggh_{c,1}}\frac{d\zeta}{R(\zeta)} & \ldots & \oint_{\ggh_{c,1}}\frac{\zeta^{N-1}d\zeta}{R(\zeta)}  \\
\ldots & \ldots & \ldots  \\
\oint_{\ggh_{c,N}}\frac{d\zeta}{R(\zeta)} & \ldots & \oint_{\ggh_{c,N}}\frac{\zeta^{N-1}d\zeta}{R(\zeta)}
\end{array}\right)^T
\left(\begin{array}{c}
\frac{\partial\vec W}{\partial\mu}\\
\frac{\partial\vec\Omega}{\partial\mu}
\end{array}\right)
=-\left(\begin{array}{c}
{\oint_{\ggh(\mu)}\frac{f_{\mu}(\zeta,\mu)}{R(\zeta,\va)}d\zeta}\\
\ldots\\
{\oint_{\ggh(\mu)}\frac{\zeta^{N-1} f_{\mu}(\zeta,\mu)}{R(\zeta,\va)}d\zeta}
\end{array}\right).
\end{equation}
So $\frac{\partial\Omega_j}{\partial\mu}$ and $\frac{\partial W_j}{\partial\mu}$ satisfy (\ref{eq:Omega_mu_thm}).
% and (\ref{eq:W_mu_thm}).
Note that $D\neq 0$ for distinct $\alpha_j$'s \cite{TV_det}.

\end{proof}

\begin{remark}
In \cite{TV_det}, there was considered the case when the contour $\gamma$ was
independent of external parameters $x$ and $t$ and simple linear dependence of $f$ on these parameters.
In this paper we apply the methods of \cite{TV_det} to
case of the dependence on parameter $\mu$ when the jump contour explicitly passes through $z=\mm$ a
point of singularity of $f$. Despite this more complicated dependence on $\mu$, the resulting formulae
are the same. The main reason is Lemma \ref{lem:par_int_zero_NLS}, which allows to find partial derivatives
with respect to $\mu$ of contour integrals involving dependence on $\mu$ in both integrands and contours
of integration.
\end{remark}

\begin{remark}
Theorem \ref{thm:perturb_NLS} guarantees that the
solution of the RHP (\ref{eq:RHP_h_NLS}) is uniquely continued
with respect to external parameters. Additional sign conditions on
$\Im h$ need to be satisfied, for $h$ to correspond to an
asymptotic solution of NLS as in \cite{TVZzero}. The sign
conditions have to be satisfied near $\gamma$ and additionally on
a semi-infinite complementary arcs connecting the arcs end points
of $\gamma$ to $\infty$.
\end{remark}

\subsection{Sign conditions and preservation of genus}

If the scalar RHP (\ref{eq:RHP_h_NLS_def}) is implemented in the asymptotic solution of the
semiclassical NLS, certain sign conditions must be satisfied.
Specifically, $\Im h(z)=0$ on $\gg_{m,j}$, $\Im h(z)<0$ on both sides of $\gg_{m,j}$,
and $\Im h(z)\ge 0$ on $\gg_{c,j}$ (see definition 3.11 below). In this section we
investigate the preservation of the sign structure of $\Im h$ under perturbations of $\mu$.

\begin{definition}
Define $\gamma^{\infty}=\gamma^{\infty}(\va,\mu)$ as an extension
of a contour $\gamma(\va,\mu)\in\Gamma(\va,\mu)$ as
$\gamma^{\infty}(\va,\mu)=(\infty,\alpha_{4N+1}]\cup\gamma(\va,\mu)\cup[\alpha_{4N},\infty)$.
Both additional arcs are considered as a complementary arc
$\gg_{c,N+1}=(\infty,\alpha_{4N+1}]\cup[\alpha_{4N},\infty)$
and assume $\gg_{c,N+1}=\overline{\gg_{c,N+1}}$, so
$\gamma^\infty=\overline{\gamma^\infty}$. With a slight abuse of
notation we write $\gamma^\infty(\va,\mu)\in\Gamma(\va,\mu)$.
\end{definition}

\begin{lemma}
If the conditions of Theorem \ref{thm:perturb_NLS} holds on
$\gamma^\infty(\va_0,\mu_0)\in\Gamma(\va_0,\mu_0)$ for
$\vec{K}(\va_0,\mu_0)=\vec{0}$ then the statement of the theorem
holds on $\gamma^\infty(\va,\mu)$, where
$\vec{K}(\va,\mu)=\vec{0}$. \label{thm:perturb_infty}
\end{lemma}

\begin{proof}$ $\\
The proof is unchanged since $f$ is analytic near the additional
semi-infinite arcs in $\gg_{c,N+1}$ and the jump condition for on
the additional complementary arc $\gg_{c,N+1}$ is taken to be zero
($\Omega_{N+1}=0$) \cite{TVZzero}.
\end{proof}

Note that the conditions in Lemma \ref{thm:perturb_infty} are
more restrictive since $\gamma\subset\gamma^\infty$.

\begin{definition} \label{def:sign_conditions}
A function $h$ satisfies sign conditions on $\gamma^\infty$ if
$\Im h(z)=0$ if $z\in\gg_{m,j}$, $\Im h(z)<0$ on both sides of
$\gamma_{m,j}$ for all $j=0,...,N$, and $\Im h(z)\ge 0$ if
$z\in\gg_{c,j}$ for all $j=1,...,N+1$. Denote $h\in
SC(\gamma^\infty)$.
\end{definition}

Note that the zero sign conditions ($\Im h(z)=0$) on $\gg_{m,j}$
are satisfied automatically through the construction of $h(z)$ by
(\ref{eq:h(z)_def_loops}) in the case of $h$ solving a RHP
(\ref{eq:RHP_h_NLS}). We only need to check preservation of
negative signs of $\Im h$ on both sides of the main arcs
$\gamma_{m,j}$ and the nonnegativity of $\Im h$ on the
complementary arcs $\gamma_{c,j}$, especially on the semi-infinite
arcs $(\infty,\alpha_{4N+1}]$ and $[\alpha_{4N},\infty)$.

\begin{remark}
Introducing the sign conditions in definition \ref{def:sign_conditions} requires to revisit
Lemma \ref{lem:gamma_in_Gamma_near} since the main arcs $\gg_{m,j}$ are now rigid (non-deformable like the complementary arcs)
due to the requirement for $\Im h$ to be negative on both sides of $\gamma(\va(\mu),\mu)$.
It has been established in \cite{KMM_book_03}
(see Lemma 5.2.1) and in \cite{TVZ2} (see Theorem 3.2), all the contours persist under deformations
of parameters $x$ and $t$ provided all sign inequalities are satisfied. This can be adapted to
the deformations of $\mu$.
The danger of a main arcs splitting into several disconnected branches as we perturb $\mu$ is adverted by the fact that
in the limit as $\mu\to\mu_0$ nonlinear local behavior would be produced near the main arcs while the condition $h'/R \neq 0$ on $\gamma$ implies linear local
behavior. Thus Lemma \ref{lem:gamma_in_Gamma_near} is valid even with the added new sign conditions.
Thus we only need to show that the sign conditions are satisfied.
\end{remark}

\begin{theorem}$ $\\
Let $f$ be defined by (\ref{eq:f-function_NLS}). Let
$\vec{K}(\va_0,\mu_0)=\vec{0}$,
$\gamma_0^\infty\in\Gamma(\va_0,\mu_0)$ and assume $h$ solves
$RHP(\gamma_0^\infty,\va_0,\mu_0,f)$ with
$\frac{h'(z,\mu_0)}{R(z,\mu_0)}\neq 0$ for all
$z\in\gamma_0^{\infty}$, and $h\in SC(\gamma_0^\infty)$.

Then there is an open neighborhood of $\mu_0$ where for all $\mu$, there is an $h$ which solves $RHP(\gamma^\infty,\va,\mu,f)$
with $\gamma^\infty=\gamma^\infty(\va,\mu)$, $\vec{K}(\va,\mu)=\vec{0}$,
$\frac{h'(z,\mu)}{R(z,\mu)}\neq 0$ for all $z\in\gamma^{\infty}$,
and $h\in SC(\gamma^\infty)$.
\label{thm:signs_NLS}
\end{theorem}

\begin{proof}$ $\\
Take any $\mu$ in a small enough open neighborhood of $\mu_0$.
There are two things we need to prove in addition to Lemma
\ref{thm:perturb_infty}: $\frac{h'(z,\mu)}{R(z,\va(\mu))}\neq 0$ on $\gamma^{\infty}$
and the sign conditions of $\Im h$ on $\gamma^{\infty}$.

Assume $\frac{h'(z,\mu_0)}{R(z,\va(\mu_0))}\neq 0$ on $\gamma_0^\infty$. Then there is a constant $C>0$
such that $\left|\frac{h'(z,\mu_0)}{R(z,\va(\mu_0))}\right|>C$ for all $z\in\gamma_0$.
Consider the solution $h(z,\mu)$ of $RHP(\gamma^\infty,\va,\mu)$, where $\vec{K}(\va,\mu)=\vec{0}$.
By Theorem \ref{thm:perturb_NLS} and Lemma \ref{thm:perturb_infty} such function exists and continuously differentiable in $\mu$.
Moreover, $h'(z,\mu)$ is continuous in $\mu$. Since $\gamma$ is a compact set in $\mathbb{C}$ and $\frac{h'(z,\mu)}{R(z,\va(\mu))}$ is continuous in $z$
and $\mu$, we have $\frac{h'(z,\mu)}{R(z,\va(\mu))}\neq 0$ for all $z\in\gamma$.

To show that $\frac{h'(z,\mu)}{R(z,\va(\mu))}\neq 0$ holds on $\gg^\infty$, we will now make use of the following properties of $f(z)$. On the real axis (in the non-tangential limit from the upper half plane),
\[
\Im f(z+i0)=\lim_{\delta\to0^+} f(z+i\delta) ,\quad\quad z\in\mathbb{R},
\]
is a piecewise linear function \cite{TVZzero}

\begin{equation}
\Im f(z+i0)= \left\{
\begin{array}{l}
\frac{\pi}{2} \left(\mm-|z|\right),\quad\quad z<\mm\\
\frac{\pi}{2} \left(z-\mm\right),\quad\quad z\ge\mm
\end{array}\right. \label{eq:Im_f_on_real_axis}
\end{equation}
and since $g(z)$ is real on the real axis, $\Im h(z+i0)=-\Im f(z+i0)$.
It is important for us that $|\Im h(z)|$ can be bounded away from zero
as $z\to\infty$.

Similarly,
\[
\Im f'(z+i0)= \left\{
\begin{array}{l}
\frac{\pi}{2} \quad\quad z\le0, \\
-\frac{\pi}{2} \quad\quad 0<z\le\mm, \\
\frac{\pi}{2} \quad\quad z>\mm,
\end{array}\right.
\]
and since $g'(z)$ is real on the real axis, $\Im h'(z+i0)=-\Im f'(z+i0)$.

Recall $\gamma^{\infty}=(\infty,\alpha_{4N+1}]\cup\ \gamma\ \cup[\alpha_{4N},\infty)$.
The semi-infinite arcs
$(\infty,\alpha_{4N+1}]$ and $[\alpha_{4N},\infty)$ can be pushed
to the real axis as
$\left(-\infty-i0,-\mm-i0\right)\cup\left[-\mm-i0,\alpha_{4N+1}\right]$
and
$\left[\alpha_{4N},-\mm+i0\right]\cup\left(-\mm+i0,-\infty+i0\right)$
respectively.

On $\left[-\mm-i0,\alpha_{4N+1}\right]$ and
$\left[\alpha_{4N},-\mm+i0\right]$,
$\frac{h'(z,\mu)}{R(z,\va(\mu))}\neq 0$ by continuity on a compact
set. Finally, for all $z\in\left(-\mm+i0,-\infty+i0\right)$, $\Im
h'(z,\mu)=-\frac{\pi}{2}$ and $R(z,\va)\in\mathbb{R}$. So
$\frac{h'(z,\mu)}{R(z,\va(\mu))}\neq 0$ for all
$z\in\left(-\mm+i0,-\infty+i0\right)$. The interval
$\left(-\infty-i0,-\mm-i0\right)$ is done similarly. So
$\frac{h'(z,\mu)}{R(z,\va(\mu))}\neq 0$ for all
$z\in\gamma^\infty$, for any $\mu$ in the neighborhood of
$\mu_0$.

Let $h\in SC(\gamma_0^\infty)$. Then $h\in SC(\gamma(\mu))$ by
continuity of $h$ in $z$ and $\mu$, compactness of $\gamma$, and
harmonicity of $\Im h$ combined with $\frac{h'(z,\mu)}{R(z,\va(\mu))}\neq 0$ for all
$z\in\gamma^\infty$ which guarantees that the (negative) signs near the main arcs $\ggh_{m,j}$ are preserved.
On the semi-infinite arcs $\left(-\infty-i0,-\mm-i0\right)$ and
$\left(-\mm+i0,-\infty+i0\right)$, $\Im h(z)=\frac{\pi}{2}\left(|z|-\frac{\mu}{2}\right)$ is positive,
 and $\left[\alpha_{4N},-\mm\right]$ and $\left[-\mm,\alpha_{4N+1}\right]$
are compact. So $\Im h \ge 0$ on $\gamma^{\infty}(\mu)$, that is $h\in SC(\gamma^\infty(\mu))$.

\end{proof}

\begin{definition}$ $\\
We define the (finite) genus $G=G(\mu)$ of the asymptotic solution of the semiclassical one
dimensional focusing NLS with the initial condition defined through $f(z,\mu)$, as (finite)
$N\in\mathbb{N}$ if there exists an asymptotic solution of NLS through the solution $h(z,\mu)$
of $RHP(\gamma^\infty,\va,\mu,f)$ with $\va=(\alpha_0,\alpha_1,...,\alpha_{4N+1})$,
such that $\frac{h'(z,\mu)}{R(z)}\neq 0$ for all $z\in\gamma^{\infty}$,
and the signs conditions of $h$ on $\gamma^\infty$ are satisfied: $h\in SC(\gamma^\infty)$.
\label{def:genus_NLS}
\end{definition}

\begin{remark}
This definition of the genus of the asymptotic solution coincides with
the genus of the (limiting) hyperelliptic Riemann surface of $R(z)$.
\end{remark}

\begin{theorem}(Preservation of genus)$ $\\
Suppose for $\mu_0$, the genus of the asymptotic solution of NLS with initial
condition defined through $f(z,\mu_0)$ in (\ref{eq:f-function_NLS}), is $G(\mu_0)$.

Then there is an open neighborhood of $\mu_0$ such that, for all $\mu$ in the neighborhood of $\mu_0$,
the genus of the asymptotic solution of NLS with initial condition defined through $f(z,\mu)$, is
preserved $G(\mu)=G(\mu_0)$.
\label{thm:pres_genus_NLS}
\end{theorem}

\begin{proof}

Follows from Theorem \ref{thm:signs_NLS} and Definition \ref{def:genus_NLS}.

\end{proof}

\begin{corollary}
Fix $x$ and $t>t_0$, where $t_0(x)$ is the time of the first break in the asymptotic solution.
Then in some open neighborhood of $\mu=2$ the genus of the solution is 2.
\end{corollary}

\begin{proof}
For $\mu=2$ and $t>t_0(x)$ the genus is 2 for all $x$ \cite{TVZzero}. By Theorem \ref{thm:pres_genus_NLS},
the genus is preserved in some open neighborhood of $\mu=2$, including
some open interval for $\mu<2$.
\end{proof}

\subsection{Numerics}

Figure \ref{fig:evolutions} demonstrates comparison solutions of
(\ref{eq:vec_K_system_NLS}) and (\ref{eq:alpha_mu_thm}) in genus 2
(see also Appendix (\ref{eq:vec_K_system_NLS_g2}) and
(\ref{eq:alpha_mu_thm_g2}) for more explicit expressions). The
solutions are practically indistinguishable on the figure with the
absolute difference less than $10^{-3}$ for $\mu\in[1,3]$, which
includes the critical value $\mu=2$, transition between
(solitonless) pure radiation case ($\mu\ge 2$) and the region with
solitons ($0<\mu<2$).

\begin{figure}
\begin{center}
\includegraphics[height=15cm]{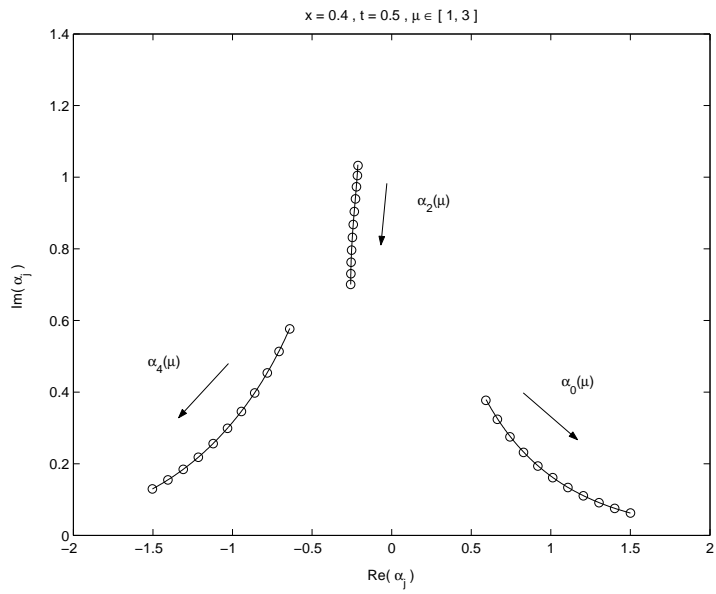}
\caption{\label{fig:evolutions}Comparison of $\mu$ evolution of
$\va=(\alpha_0, \alpha_2,\alpha_4)$ using
(\ref{eq:vec_K_system_NLS_g2}) (solid lines) and (\ref{eq:alpha_mu_thm_g2}) (circles).}
\end{center}
\end{figure}

\section{Appendix}

\subsection{Genus $0$ region}

\begin{figure}
\begin{center}
\includegraphics[height=5cm]{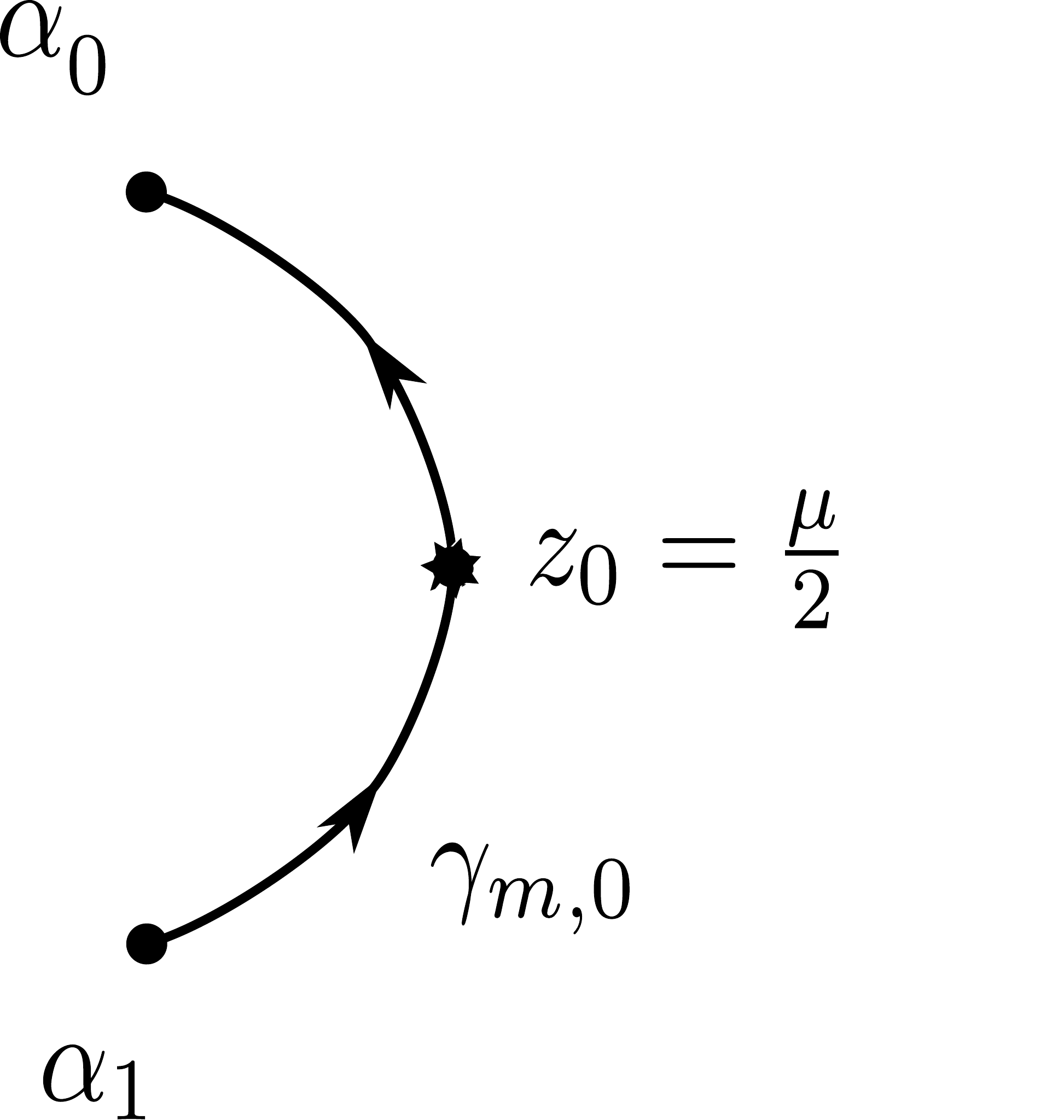}
\includegraphics[height=5cm]{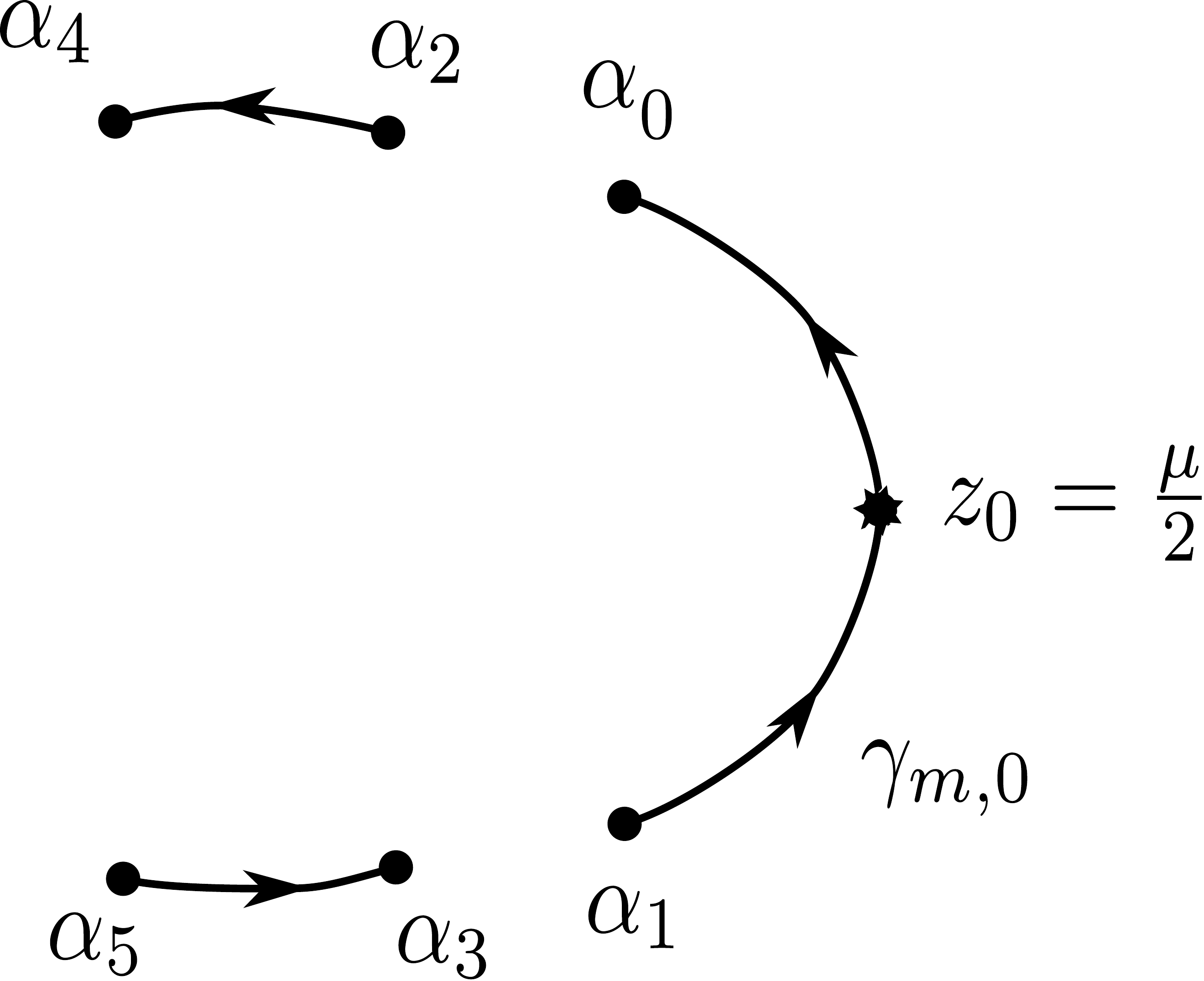}
\caption{\label{fig:NLS_genus_2_contour} The jump contour in the case of genus 0 and genus 2 with complex-conjugate symmetry in the notation of \cite{TVZzero}.}
\end{center}
\end{figure}

It was shown in \cite{TVZzero} that for all $\mu>0$ and for all $x$, there is a breaking curve $t=t_0(x)$
in the $(x,t)$ plane. The region $0\le t<t_0(x)$ has genus $0$ in the sense of genus of the underlying
Riemann surface for the square root
\[
R(z,\alpha_0)=\sqrt{(z-\alpha_0)(z-{\alpha}_1)}, \quad \alpha_1=\overline{\alpha}_0
\]
where the branchcut is chosen along the main arc connecting $\alpha_0$ and $\alpha_1=\overline{\alpha}_0$ through $z=\mm$,
and the branch is fixed by $R(z)\to -z$ as $z\to+\infty$.
The asymptotic solution of NLS is expressed in terms of $\alpha_0=\alpha_0(x,t,\mu)$.

All expressions in the genus $0$ region ($N=0$) have simpler form. In particular:
\begin{equation}
h(z,\alpha_0,\mu)=
\frac{R(z,\alpha_0)}{2\pi i}\oint_{\ggh(\mu)}\frac{f(\zeta,\mu)d\zeta}{(\zeta-z)R(\zeta,\alpha_0)},
\end{equation}
\begin{equation}
K(z,\alpha_0,\mu)=
\frac{1}{2\pi i}\oint_{\ggh(\mu)}\frac{f(\zeta,\mu)d\zeta}{(\zeta-z)R(\zeta,\alpha_0)},
\end{equation}
and with a slight abuse of notation
\begin{equation}
K(\alpha_0,\mu):=K(\alpha_0,\alpha_0,\mu)=
\frac{1}{2\pi i}\oint_{\ggh(\mu)}\frac{f(\zeta,\mu)d\zeta}{(\zeta-\alpha_0)R(\zeta,\alpha_0)},
\end{equation}
and
\begin{equation}
\frac{\partial K}{\partial\mu}(\alpha_0,\mu):=\frac{\partial K}{\partial\mu}(\alpha_0,\alpha_0,\mu)=
\frac{1}{2\pi i}\oint_{\ggh(\mu)}\frac{f_\mu(\zeta,\mu)d\zeta}{(\zeta-\alpha_0)R(\zeta,\alpha_0)}
\end{equation}

\begin{theorem} ($\mu$-perturbation in genus $0$)\\
Consider a finite length non-selfintersecting oriented arc $\gamma_0=[\alpha_0(\mu_0),\overline{\alpha}_0(\mu_0)]\in\Gamma(\va,\mu_0)$ in the complex plane
 with the distinct end points ($\alpha_0\neq\overline{\alpha}_0$) and depending on a parameter $\mu$ (see Figure \ref{fig:NLS_genus_2_contour}).
 Assume $\alpha_0$ and $\mu_0$ satisfy the equation
\[
{K}\left(\alpha_0,\mu_0\right)={0},
\label{eq:K_0_g0}
\]
and $f$ is given by (\ref{eq:f-function_NLS}).
Let $\gamma=\gamma(\va,\mu)$ be the contour of a RH problem which seeks a function
$h(z)$ which satisfies the following conditions
\begin{equation}
\left\{
\begin{array}{l}
h_{+}(z)+h_{-}(z)=0,\ \mbox{on} \ \gg,\\
h(z)+f(z)\ \mbox{is analytic in}\ \mathbb{\overline{C}}\backslash\gg,
\end{array}
\right. \label{eq:RHP_h_NLS_g0}
\end{equation}
Assume that there is a function $h(z,\alpha_0,\mu_0)$ which satisfies (\ref{eq:RHP_h_NLS_g0}) and suppose $\frac{h'(z,\alpha_0,\mu_0)}{R(z,\alpha_0)}\neq 0$
for all $z$ on $\gamma_0$.

Then there is a contour $\gamma(\va,\mu)\in\Gamma(\va,\mu)$ such that the solution $\alpha_0(\mu)$ of the equation
\begin{equation}
{K}\left(\alpha_0,\mu\right)={0} \label{eq:vec_K_system_NLS_g0}
\end{equation}
and $h(z,\alpha(\mu),\mu)$ which solves (\ref{eq:RHP_h_NLS_g0}) are uniquely defined and continuously differentiable in $\mu$ in some
open neighborhood of $\mu_0$.

Moreover,
\begin{equation}
\frac{\partial\alpha_0}{\partial\mu}(\mu)=-\frac{2\pi i
\frac{\partial K}{\partial\mu}(\alpha_0(\mu),\mu)}{
\oint_{\ggh(\mu)}\frac{f'(\zeta,\mu)}{(\zeta-\alpha_0(\mu))R(\zeta,\alpha_0(\mu))}d\zeta},
\label{eq:alpha_mu_thm_g0}
\end{equation}
\begin{equation}
\frac{\partial h}{\partial\mu}(z,\mu)=\frac{R(z,\alpha_0(\mu))}{2\pi i}\oint_{\ggh(\mu)}
\frac{\frac{\partial f}{\partial\mu}(\zeta,\mu)}{(\zeta-z)R(\zeta,\alpha_0(\mu))}d\zeta,
\end{equation}
where $z$ is inside of $\ggh$.
\label{thm:perturb_NLS_g0}
\end{theorem}

\subsection{Genus $2$ region}

In the genus 2 region ($N=2$),
with underlying Riemann surface for the square root
\[
R(z)=\sqrt{(z-\alpha_0)(z-{\alpha}_1)(z-\alpha_2)(z-{\alpha}_3)(z-\alpha_4)(z-{\alpha}_5)},
\]
where the branchcut is chosen along the main arcs connecting $\alpha_0$ and $\alpha_1$,
$\alpha_2$ and $\alpha_4$, $\alpha_5$ and $\alpha_3$;
and the branch is fixed by $R(z)\to -z^3$ as $z\to+\infty$.

Taking into account the complex-conjugate symmetry
\begin{equation}
\alpha_{1}=\overline{\alpha}_{0}, \ \alpha_{3}=\overline{\alpha}_{2}, \ \alpha_{5}=\overline{\alpha}_{4}.
\end{equation}
\begin{equation}
h(z)=\frac{R(z)}{2\pi i}
\left[\oint_{\ggh}\frac{f(\zeta)}{(\zeta-z)R(\zeta)}d\zeta +
\oint_{\ggh_{m}}\frac{W}{(\zeta-z)R(\zeta)}d\zeta+
\oint_{\ggh_{c}}\frac{\Omega}{(\zeta-z)R(\zeta)}d\zeta \right],
\label{eq:h(z)_g2}
\end{equation}
where $z$ is inside of $\ggh$, $\ggh_m$ is a loop around the main arc $\gg_m=[\alpha_2,\alpha_4]\cup[\alpha_5,\alpha_3]$,
and $\ggh_c$ is a loop around the complementary arc $\gg_c=[\alpha_0,\alpha_2]\cup[\alpha_3,\alpha_1]$
(see Fig. \ref{fig:NLS_genus_2_contour}). Real constants $W$ and $\Omega$ solve the system
\begin{equation}
\left\{
\begin{array}{l}
\oint_{\ggh}\frac{f(\zeta)}{R(\zeta)}d\zeta +
\Omega\oint_{\ggh_{c}}\frac{d\zeta}{R(\zeta)}+
W\oint_{\ggh_{m}}\frac{d\zeta}{R(\zeta)} =0,\\
\oint_{\ggh}\frac{\zeta f(\zeta)}{R(\zeta)}d\zeta +
\Omega\oint_{\ggh_{c}}\frac{\zeta }{R(\zeta)}d\zeta+
W\oint_{\ggh_{m}}\frac{\zeta }{R(\zeta)}d\zeta =0.
\end{array}\right.
\end{equation}

Other useful expressions written explicitly in genus 2 region
\begin{equation}
K(z)=\frac{1}{2\pi i}\left|\begin{array}{ccc}
\oint_{\ggh_{m}}\frac{d\zeta}{R(\zeta)} & \oint_{\ggh_{m}}\frac{\zeta d\zeta}{R(\zeta)}
& \oint_{\ggh_{m}}\frac{d\zeta}{(\zeta-z)R(\zeta)} \\
\oint_{\ggh_{c}}\frac{d\zeta}{R(\zeta)} & \oint_{\ggh_{c}}\frac{\zeta d\zeta}{R(\zeta)}
& \oint_{\ggh_{c}}\frac{d\zeta}{(\zeta-z)R(\zeta)} \\
\oint_{\ggh}\frac{f(\zeta)d\zeta}{R(\zeta)} & \oint_{\ggh}\frac{\zeta f(\zeta)d\zeta}{R(\zeta)} &
\oint_{\ggh}\frac{f(\zeta)d\zeta}{(\zeta-z)R(\zeta)}
\end{array}\right|\\
\label{eq:K(z)_det_g2}
\end{equation}
or
\begin{equation}
K(z)=\frac{1}{2\pi i}
\left[\oint_{\ggh}\frac{f(\zeta)}{(\zeta-z)R(\zeta)}d\zeta +
\oint_{\ggh_{m}}\frac{W}{(\zeta-z)R(\zeta)}d\zeta+
\oint_{\ggh_{c}}\frac{\Omega}{(\zeta-z)R(\zeta)}d\zeta \right],
\label{eq:K(z)_g2}
\end{equation}
where $z$ is inside of $\ggh$.

\begin{equation}
\frac{\partial K}{\partial\mu}(\alpha_j,\va,\mu)=
\frac{1}{2\pi i}\left|\begin{array}{cccc}
\oint_{\ggh_{m}}\frac{d\zeta}{R(\zeta)} & \oint_{\ggh_{m}}\frac{\zeta d\zeta}{R(\zeta)}
& \oint_{\ggh_{m}}\frac{d\zeta}{(\zeta-\alpha_j)R(\zeta)} \\
\oint_{\ggh_{c}}\frac{d\zeta}{R(\zeta)} & \oint_{\ggh_{c}}\frac{\zeta d\zeta}{R(\zeta)}
& \oint_{\ggh_{c}}\frac{d\zeta}{(\zeta-\alpha_j)R(\zeta)} \\
\oint_{\ggh}\frac{f_\mu(\zeta)d\zeta}{R(\zeta)} & \oint_{\ggh}\frac{\zeta f_\mu(\zeta)d\zeta}{R(\zeta)} &
\oint_{\ggh}\frac{f_\mu(\zeta)d\zeta}{(\zeta-\alpha_j)R(\zeta)}
\end{array}\right|
\label{eq:DK/Dmu_det_NLS_g2}
\end{equation}
or
\begin{equation}
\frac{\partial K}{\partial\mu}(\alpha_j,\va,\mu)=
\frac{1}{2\pi i}
\left[\oint_{\ggh}\frac{f_{\mu}(\zeta)}{(\zeta-\alpha_j)R(\zeta)}d\zeta +
\oint_{\ggh_{m}}\frac{W_{\mu}}{(\zeta-\alpha_j)R(\zeta)}d\zeta+
\oint_{\ggh_{c}}\frac{\Omega_{\mu}}{(\zeta-\alpha_j)R(\zeta)}d\zeta \right],
\label{eq:DK/Dmu_int_NLS_g2}
\end{equation}
where $f_\mu$ is given by (\ref{eq:f'_mu}). Real constants $W_\mu$ and $\Omega_\mu$ solve the system
\begin{equation}
\left\{
\begin{array}{l}
\oint_{\ggh}\frac{f_\mu(\zeta)}{R(\zeta)}d\zeta +
\Omega_\mu\oint_{\ggh_{c}}\frac{d\zeta}{R(\zeta)}+
W_\mu\oint_{\ggh_{m}}\frac{d\zeta}{R(\zeta)} =0,\\
\oint_{\ggh}\frac{\zeta f_\mu(\zeta)}{R(\zeta)}d\zeta +
\Omega_\mu\oint_{\ggh_{c}}\frac{\zeta }{R(\zeta)}d\zeta+
W_\mu\oint_{\ggh_{m}}\frac{\zeta }{R(\zeta)}d\zeta =0.
\end{array}\right.
\end{equation}
\begin{equation}
D=\left|\begin{array}{cc}
\oint_{\ggh_{m}}\frac{d\zeta}{R(\zeta)} & \oint_{\ggh_{m}}\frac{\zeta d\zeta}{R(\zeta)} \\
\oint_{\ggh_{c}}\frac{d\zeta}{R(\zeta)} & \oint_{\ggh_{c}}\frac{\zeta d\zeta}{R(\zeta)}
\end{array}\right|.
\end{equation}

\begin{theorem} ($\mu$-perturbation in genus $2$)\\
Consider a finite length non-selfintersecting contour $\gamma_0$ in the complex plane consisting of a union of oriented arcs  $\gamma_0=\gg_{m}\cup\gg_{c}\cup[\alpha_0,\overline{\alpha}_0]$
 with the distinct arcs end points $\va_0=(\alpha_0,\alpha_2,\alpha_4)$ in the upper half plane and depending on a parameter $\mu$ (see Figure \ref{fig:NLS_genus_2_contour}).
Assume $\va_0$ and $\mu_0$ satisfy a system of equations
\[
\left\{\begin{array}{l}
K\left(\alpha_0,\va_0,\mu_0\right)={0},\\
K\left(\alpha_2,\va_0,\mu_0\right)={0},\\
K\left(\alpha_4,\va_0,\mu_0\right)={0},
\end{array}\right.
\label{eq:K_0_g2}
\]
and $f$ is given by (\ref{eq:f-function_NLS}).
Let $\gamma=\gamma(\va,\mu)$ be the contour of a RH problem which seeks a function
$h(z)$ which satisfies the following conditions
\begin{equation}
\left\{
\begin{array}{l}
h_{+}(z)+h_{-}(z)=0,\ \mbox{on} \ \gg_{m,0}=[\alpha_0,\overline{\alpha}_0],\\
h_{+}(z)+h_{-}(z)=2W,\ \mbox{on} \ \gg_m,\\
h_{+}(z)-h_{-}(z)=2\Omega,\ \mbox{on} \ \gg_c,\\
h(z)+f(z)\ \mbox{is analytic in}\ \mathbb{\overline{C}}\backslash\gg,
\end{array}
\right. \label{eq:RHP_h_NLS_g2}
\end{equation}
where $\Omega=\Omega(\va,\mu)$ and $W=W(\va,\mu)$ are real constants whose numerical values will be determined from the RH conditions.
Assume that there is a function $h(z,\va_0,\mu_0)$ which satisfies (\ref{eq:RHP_h_NLS_g2}) and suppose $\frac{h'(z,\va_0,\mu_0)}{R(z,\va_0)}\neq 0$
for all $z$ on $\gamma_0$.

Then there is a contour $\gamma(\va,\mu)\in\Gamma(\va,\mu)$ such that the solution $\va=\va(\mu)$ of the system
\begin{equation}
\left\{\begin{array}{l}
K\left(\alpha_0,\va,\mu\right)={0},\\
K\left(\alpha_2,\va,\mu\right)={0},\\
K\left(\alpha_4,\va,\mu\right)={0},
\end{array}\right.
\label{eq:vec_K_system_NLS_g2}
\end{equation}
and $h(z,\va(\mu),\mu)$ which solves (\ref{eq:RHP_h_NLS_g2}) are uniquely defined and continuously differentiable in $\mu$ in some
neighborhood of $\mu_0$.

Moreover,
\begin{equation}
\left\{
\begin{array}{l}
\frac{\partial\alpha_0}{\partial\mu}(x,t,\mu)=-\frac{2\pi i\
\frac{\partial
K}{\partial\mu}(\alpha_0,\va,\mu)}{D
\oint_{\ggh(\mu)}\frac{f'(\zeta)}{(\zeta-\alpha_0)R(\zeta)}d\zeta},\\
\quad \\
\frac{\partial\alpha_2}{\partial\mu}(x,t,\mu)=-\frac{2\pi i\
\frac{\partial
K}{\partial\mu}(\alpha_2,\va,\mu)}{D
\oint_{\ggh(\mu)}\frac{f'(\zeta)}{(\zeta-\alpha_2)R(\zeta)}d\zeta},\\
\quad \\
\frac{\partial\alpha_4}{\partial\mu}(x,t,\mu)=-\frac{2\pi i\
\frac{\partial
K}{\partial\mu}(\alpha_4,\va,\mu)}{D
\oint_{\ggh(\mu)}\frac{f'(\zeta)}{(\zeta-\alpha_4)R(\zeta)}d\zeta},
\end{array}
\right.
\label{eq:alpha_mu_thm_g2}
\end{equation}
\begin{equation}
\frac{\partial h}{\partial\mu}(z,x,t,\mu)=\frac{R(z)}{2\pi i}\oint_{\ggh(\mu)}
\frac{\frac{\partial f}{\partial\mu}(\zeta)}{(\zeta-z)R(\zeta)}d\zeta,
\end{equation}
where $z$ is inside of $\ggh$,

Furthermore, $\Omega(\mu)=\Omega(\va(\mu),\mu)$, and $W(\mu)=W(\va(\mu),\mu)$
are defined and continuously differentiable in $\mu$ in some open neighborhood of $\mu_0$, and
\begin{equation}
\frac{\partial\Omega}{\partial\mu}(x,t,\mu)=
-\frac{1}{D}
\left|\begin{array}{cc}
\oint_{\ggh_{m}}\frac{d\zeta}{R(\zeta)} & \oint_{\ggh_{m}}\frac{\zeta d\zeta}{R(\zeta)} \\
{\oint_{\ggh}\frac{ f_{\mu}(\zeta)}{R(\zeta)}d\zeta} & {\oint_{\ggh}\frac{\zeta  f_{\mu}(\zeta)}{R(\zeta)}d\zeta}
\end{array}\right|,
\label{eq:O_mu_thm_g2}
\end{equation}
\begin{equation}
\frac{\partial W}{\partial\mu}(x,t,\mu)=
-\frac{1}{D}
\left|\begin{array}{cc}
{\oint_{\ggh}\frac{ f_{\mu}(\zeta)}{R(\zeta)}d\zeta} & {\oint_{\ggh}\frac{\zeta  f_{\mu}(\zeta)}{R(\zeta)}d\zeta}\\
\oint_{\ggh_{c}}\frac{d\zeta}{R(\zeta)} & \oint_{\ggh_{c}}\frac{\zeta d\zeta}{R(\zeta)}
\end{array}\right|,
\label{eq:W_mu_thm_g2}
\end{equation}
where $\alpha_j=\alpha_j(x,t,\mu)$, $R(\zeta)=R(\zeta,\va(x,t,\mu))$, $f(\zeta)=f(\zeta,x,t,\mu)$, $f_\mu(\zeta)=\frac{\partial f}{\partial\mu}(\zeta,x,t,\mu)$, and
$D=D(\va(x,t,\mu))$.
\label{thm:perturb_NLS_g2}
\end{theorem}

\end{document}